\documentclass[reqno]{amsart}
\usepackage{xcolor,amssymb}
\usepackage[margin=1in]{geometry}
\usepackage[T1]{fontenc}
\usepackage[utf8]{inputenc}
\usepackage[colorlinks=true, pdfstartview=FitV, linkcolor=blue,
citecolor=blue, urlcolor=blue]{hyperref}
\DeclareMathOperator*{\esssup}{ess\,sup}
\def\loc{\mathrm{loc}}
\DeclareMathOperator*{\diam}{diam}

\DeclareMathOperator*{\data}{data}
\newtheorem{proposition}{Proposition}
\newtheorem{lemma}{Lemma}
\newtheorem{remark}{Remark}
\newtheorem{definition}{Definition}
\newtheorem{theorem}{Theorem}
\newtheorem{corollary}{Corollary}
\numberwithin{equation}{section}

\newcommand{\dd}{\mathrm{d}}

\def\Xint#1{\mathchoice
	{\XXint\displaystyle\textstyle{#1}}%
	{\XXint\textstyle\scriptstyle{#1}}%
	{\XXint\scriptstyle\scriptscriptstyle{#1}}%
	{\XXint\scriptscriptstyle\scriptscriptstyle{#1}}%
	\!\int}
\def\XXint#1#2#3{{\setbox0=\hbox{$#1{#2#3}{\int}$ }
		\vcenter{\hbox{$#2#3$ }}\kern-.6\wd0}}

\def\dashint{\Xint-}

\title{Boundary regularity for quasiminima of double-phase problems on metric spaces}
\author{Antonella Nastasi$^1$, Cintia Pacchiano Camacho$^2$}
\date{\today}

\address{$^1$Department of Engineering, University of Palermo, Viale delle Scienze, 90128, Palermo, Italy\\
 ORCID ID: 0000-0003-1589-2235}
\email{antonella.nastasi@unipa.it}

\address{$^2$Instituto de Matem\'aticas, Unidad Cuernavaca, Universidad Nacional Aut\'onoma de M\'exico, Av. Universidad, 62210, Cuernavaca, Morelos, Mexico\\
ORCID ID: 0009-0004-6210-4013}
\email{cintia.pacchiano@im.unam.mx}



\begin{document}

\maketitle

\makeatletter
\vspace{-2em}
{\centering\enddoc@text}
\let\enddoc@text\empty 
\makeatother

\begin{abstract}
We give a sufficient condition for H\"older continuity at a boundary point for quasiminima of double-phase functionals of $p,q$-Laplace type, in the setting of metric measure spaces equipped with a doubling measure and supporting a Poincaré inequality. We use a purely analytic variational approach based on De Giorgi-type conditions to give a pointwise estimate near a boundary point. The proofs rely on a careful phase analysis and estimates in the intrinsic geometries.
\end{abstract}

\tableofcontents

\section{Introduction}
This paper aims to obtain boundary regularity of quasiminima for the double-phase integral defined on a complete metric measure space \((X, d, \mu)\), equipped with a metric \(d\) and a doubling measure \(\mu\), which supports a weak \((1,p)\)-Poincar\'e inequality. Let \(\Omega\) be an open and bounded subset of \(X\). The double-phase integral is given by
\begin{equation}\label{J}
\int_{\Omega} H(x,g_u) \, \mathrm{d}\mu = \int_{\Omega} \big(g_u^p + a(x) g_u^q\big) \, \mathrm{d}\mu,
\end{equation}
where \(g_u\) is the minimal \(p\)-weak upper gradient of \(u\) and \(1<p<q<Q\), with $Q$ generalization of the euclidean concept of space dimension (see Section \ref{Sec2}). The modulating coefficient function \(a(\cdot) \geq 0\) is assumed to meet certain standard regularity conditions, which are detailed in \eqref{aalpha} below. The integral \eqref{J} shows an energy density that alternates between two types of degenerate behaviour, depending on the modulating coefficient \(a(\cdot)\), which dictates the phase. Specifically, when \(a(x) = 0\), the variational integral \eqref{J} simplifies to the classical \(p\)-growth problem. In contrast, when \(a(x) \geq c > 0\), it corresponds to the \((p,q)\)-problem. Additionally, the ratio of exponents \(q/p\) is subjected to the condition
\begin{equation}\label{pqcond}
\frac{q}{p} \leq 1 + \frac{\alpha}{Q},
\end{equation}
where \(0 < \alpha \leq 1\), and \(Q = \log_2 C_D\), with \(C_D\) being the doubling constant of the measure. This condition is known to be sharp for ensuring regularity of minima already in the Euclidean setting, in this direction see the recent contributions of Baroni-Colombo-Mingione \cite{BCM} and Colombo-Mingione \cite{CM1, CM}. \\

\noindent The main advantage of the notion of quasiminima of \eqref{J} is that it simultaneously covers a wide class of problems 
where the variational integrand $F:\Omega\times\mathbb R\times \mathbb R\to\mathbb R$ satisfies the Carath\'eody conditions and
\[
\lambda H(x,z)\le F(x,u,z)\le\Lambda H(x,z),
\quad 0<\lambda<\Lambda<\infty,
\]
for every $x \in \Omega$ and $u,z\in \mathbb{R}$.
Originally developed by Giaquinta and Giusti \cite{GG1,GG2}, quasiminima offer a flexible framework that extends beyond traditional differentiable structures, thus providing a unified treatment of variational integrals, elliptic equations and systems, obstacle problems and quasiregular mappings. When defined over a metric measure space, quasiminima allow the analysis of variational problems in settings where classical derivatives may not be defined, relying instead on notions such as minimal weak upper gradients to understand how functions vary along paths.
Those notions permit to overcome the difficulties due to working on general
metric spaces, which provide a relaxed geometric setting and lack a smooth structure.
 This adaptability has opened the application of regularity theory to non-Euclidean spaces and various areas of analysis, such as weighted Sobolev spaces, calculus on Riemannian manifolds, and potential theory on graphs, to name a few, see Bj\"orn-Bj\"orn \cite{BB}, Franchi-Haj\l asz-Koskela \cite{FHK}, Haj\l asz \cite{H}, Heinonen-Koskela-Shanmugalingam-Tyson \cite{HKST}, Kilpel\"ainen-Kinnunen-Martio \cite{KKM}, Liu-Zhou-Shanmugalingam \cite{LSZ}, Shanmugalingam \cite{Sh1} and references therein.  
The study of variational problems on metric measure spaces originated from independent proofs by Grigor’yan \cite{GAA} and Saloff-Coste \cite{SC}, which show that, on Riemannian manifolds, the doubling property and the Poincaré inequality are equivalent to a specific Harnack-type inequality for solutions of the heat equation. However, our focus changes from Riemannian manifolds to more general spaces. \\

\noindent Recently, the study of Sobolev spaces without a differentiable structure, along with the variational theory of different energy functionals in metric measure spaces, has captured the interest of researchers. For example, Cheeger \cite{C} studied differentiability properties, while Kinnunen-Martio \cite{KM} focused on aspects of potential theory. Additionally,  Kinnunen-Shanmugalingam \cite{KS} established local properties of quasiminima for the $p$-energy integral. Bj\"{o}rn-Bj\"{o}rn-Shanmugalingam \cite{BBS} explored the Dirichlet problem for $p$-harmonic functions. Furthermore, Kinnunen-Marola-Martio \cite{KMM} proved the Harnack principle. Recently, for the double-phase problems, Kinnunen-Nastasi-Pacchiano Camacho \cite{KNP} established higher integrability results for the gradient of quasiminima and Nastasi-Pacchiano Camacho \cite{PCN2} derived other types of local regularity results.  On the other hand, the regularity theory of nonlinear parabolic problems in the metric space context has been developed and studied in Her\'an \cite{Heran1, Heran2}, Ivert-Marola-Masson \cite{IMM}, Marola-Masson \cite{MM}, Masson-Miranda Jr.-Paronetto-Parviainen \cite{MMPP} and Masson-Siljander \cite{MS}. We also report some recent contributions concearning the study of the fractional $p$-Laplacian in the context of general metric measure spaces, see Capogna-Kline-Korte-Shanmugalingam-Snipes \cite{CKKSS}.\\

\noindent Double-phase integrals have been introduced by Zhikov in the context of Homogenization and Lavrentiev's phenomenon \cite{Z1,Z2,Z3}. Regularity theory for double-phase integrals has been extensively developed, starting from the seminal works of Marcellini \cite{Ma1, Ma2, Ma3}. In the Euclidean context, Colombo-Mingione \cite{CM1,CM}, Cupini-Marcellini-Mascolo \cite{CMM}, De Filippis-Mingione \cite{DFM}, Esposito-Leonetti-Mingione \cite{ELM}, Marcellini-Nastasi-Pacchiano Camacho \cite{MNPC}, Mingione \cite{Min}, Ok \cite{O} considered classes of nonuniformly elliptic problems and proved different regularity properties. For example, H\"{o}lder regularity results for certain class of double-phase problems with non-standard growth conditions can be found in D\"{u}zg\"{u}n-Marcellini-Vespri \cite{DMV1, DMV2}, Eleuteri \cite{Ele}, Harjulehto-H\"{a}st\"{o}-Toivanen \cite{HHT}, Di Benedetto-Gianazza-Vespri \cite{DGV}. Moreover, Di Benedetto-Trudinger \cite{DT} and Baroni-Colombo-Mingione \cite{BCM} independently obtained Harnack inequalities for double-phase problems. See also Kinnunen-Lehrb\"{a}ck-V\"{a}h\"{a}kangas \cite{KLM} and Marcellini \cite{Ma0, Ma} for other results and literature reviews.\\

\noindent However, the analysis of boundary regularity, especially in metric measure spaces, remains relatively underdeveloped. Existing works include boundary regularity results for minimizers and quasiminima in the Euclidean case, such as the Wiener-type condition proved by Ziemer \cite{Ziemer}, see also Tachikawa \cite{Tach} for double-phase functionals and Irving-Koch \cite{IK} for $(p,q)$-growth functionals. For the generalization to the metric setting, we have the works of Bj\"orn \cite{B} and Nastasi-Pacchiano Camacho \cite{NP}. Further contributions include boundary oscillation estimates by Bj\"orn, MacManus, and Shanmugalingam \cite{BMS} for \(p\)-harmonic functions and \(p\)-energy minimizers. Despite these advances, the boundary behaviour of double-phase problems remains a relatively open problem. Addressing this gap, the present work studies quasiminima of \eqref{J} in the Newtonian--Sobolev space \(N^{1,1}(X)\), with \(H(\cdot,g_u) \in L^1(X)\), for $X$ a general metric measure space. The lack of classical derivatives in metric measure spaces requires variational methods that focus on the modulus of the gradient, therefore avoiding the need for smoothness.\\

\noindent To the best of our knowledge, boundary regularity for double-phase problems remains largely unexplored, even in the Euclidean setting. Motivated by recent developments, such as the works of Tachikawa \cite{Tach} and Bj\"orn \cite{B}, we aim to extend regularity results up to the boundary within the general framework of metric measure spaces, where no smooth structure is assumed. A central aspect of our approach is the analysis of functionals with non-standard growth, particularly the transition between the \(p\)- and \((p,q)\)-phases governed by the coefficient function \(a(x)\). A key novelty of this work is the introduction of a Maz'ya-type estimate for frozen functionals (Theorem~\ref{Mazya_Frozen}), inspired by classical pointwise capacitary estimates due to Maz'ya \cite{Mazya1, Mazya2}. As far as we know, such a result has not been previously established, even in the Euclidean case. This estimate enables us to generalize local regularity results, such as those in De Filippis–Mingione \cite{DFM}, to boundary points.\\

\noindent The present work is organized as follows: we first prove the local boundedness of quasiminima (Section~\ref{local boundedness}), then examine the behavior of frozen functionals (Section~\ref{frozen}), and finally derive a pointwise boundary estimate (Proposition~\ref{pointwise-estimate}, Section~\ref{Pointwise estimate on a boundary point}). This leads to the main result of the paper: a sufficient condition for the H\"older continuity of quasiminima at boundary points (Theorem~\ref{Theorem2.12}, Section~\ref{Sufficient condition for Holder continuity}). Throughout, we aim to provide detailed and self-contained arguments that contribute to the regularity theory of variational integrals under double-phase growth.

\section{Preliminaries}\label{Sec2}
This section collects definitions and results needed when working in the general context of metric measure spaces. Further details can be found in the book by Bj\"orn-Bj\"orn \cite{BB}.
Let $(X, d, \mu)$ be a complete metric measure space with a metric $d$ and a Borel regular measure $\mu$. 
The measure $\mu$ is supposed to be doubling, meaning that there exists a constant $C_D \geq 1$ such that, for every ball $B_r$ in $X$,
\begin{equation}\label{doubling}
0<\mu(B_{2r})\leq C_D \mu(B_r)<\infty.
\end{equation} 
We denote by $B_r=B(x,r)=\{x\in X:d(y,x)<r\}$ an open ball centered in $x\in X$ and with radius $0<r<\infty$. 
The following lemma introduces $Q$, which is a generalization of the usual concept of dimension in metric spaces supporting a doubling measure. Indeed, in the Euclidean $n$-space with the Lebesgue measure, we have $Q=n$.

\begin{lemma}[\cite{BB}, Lemma 3.3]\label{ineqDoubling}
	Let $(X, d, \mu)$ be a metric measure space with a doubling measure $\mu$. Then 
	\begin{equation}\label{s}
		\frac{\mu(B(y,r))}{\mu(B(x,R))}\geq C\left(\frac{r}{R}\right)^Q,
	\end{equation}
	for every $0<r\le R<\infty$, $x \in X$ and $y \in B(x,R)$. Here $Q=\log_2C_D$ and $C=C_D^{-2}$.
\end{lemma}

\noindent Let $a:X\to[0,\infty)$ be the coefficient function in \eqref{J} and let $\delta_{\mu}$ be a quasi-distance defined by 
\[
\delta_{\mu}(x,y)=\bigl(\mu(B(x,d(x,y)))+\mu(B(y,d(x,y)))\bigr)^{1/Q},\quad x,y \in X,\, x\ne y,
\]
with $Q=\log_2C_D$ as in \eqref{s} and $\delta_{\mu}(x,x)=0$. We suppose that there exists $0<\alpha\le1$ such that
\begin{equation}\label{aalpha}
	[a]_{\alpha}= \sup_{x,y \in X, x\neq y} \dfrac{|a(x)-a(y)|}{\delta_{\mu}(x,y)^{\alpha}}<\infty.
\end{equation} 
In the next remark, we point out that under certain regularity conditions on the measure, the quasi-distance $\delta_{\mu}$ and the usual distance $d$ are equivalent.

\begin{remark}
A measure is called Ahlfors--David regular, if there exist constants $0<C_2\le C_1<\infty$ such that
\begin{equation*}\label{ahlfors}
C_2r^Q\le\mu(B(x,r))\le C_1r^Q,
\end{equation*}
for every $x\in X$ and $0<r\le\diam(X)$.
If the measure $\mu$ is Ahlfors--David regular, then $\delta_{\mu}(x,y)\approx d(x,y)$ for every $x,y$ and, consequently,
$[a]_\alpha<\infty$ if and only if $a$ is H\"older continuous with the exponent $\alpha$.
\end{remark}

\noindent Throughout this manuscript, we suppose that $\mu$ is upper $Q$-Ahlfors regular, that is there exists a constant $C_1>0$ such that our measure satisfies the following inequality
\begin{equation}\label{upper Q-Ahlfors}
\mu(B(x,r))\leq C_1 r^Q \quad \mbox{for every $x\in X$ and $0<r\leq {\rm diam(X)}.$}
\end{equation}
This assumption ensures some uniformity and regularity in the distribution of the measure. We note that this is a general assumption. For example, self-similar fractals, metric measure spaces with controlled curvature, uniformly rectifiable sets and Carnot groups can exhibit upper Ahlfors regularity. \\

\noindent We discuss the notion of upper gradient as a way to generalize the modulus of the gradient in the Euclidean case to the metric setting. 

\begin{definition}\label{uppergradient}
	A nonnegative Borel function $g$ is said to be an upper gradient of function $u: X \to [-\infty,\infty]$ if, for all paths $\gamma$ connecting $x$ and $y$, we have 
	\begin{equation*}
		|u(x)-u(y)|\leq \int_{\gamma}g\, \dd s, 
	\end{equation*}
	whenever $u(x)$ and $u(y)$ are both finite and $\int_{\gamma}g \, \dd s= \infty$ otherwise. Here $x$ and $y$ are the endpoints of $\gamma$.
	Moreover, if a nonnegative measurable function $g$ satisfies the inequality above for $p$-almost every path, 
	that is, with the exception of a path family of zero $p$-modulus, then $g$ is called a $p$-weak upper gradient of $u$.
\end{definition}
\noindent We note that if $u$ has an upper gradient $g\in L^p(X)$, there exists a unique minimal $p$-weak upper gradient $g_u\in L^p(X)$ with
$g_u\le g$ $\mu$-almost everywhere for all $p$-weak upper gradients $g\in L^p(X)$ of $u$, see  \cite[Theorem 2.5]{BB}.\\

\noindent Let $\Vert u\Vert_{N^{1,p}(X)}=\Vert u\Vert_{L^{p}(X)}+\Vert g_u\Vert_{L^{p}(X)}$.
Consider the collection of functions $u\in L^p(X)$ with an upper gradient $g\in L^p(X)$ and let
\begin{equation*}
	\widetilde{N}^{1,p}(X)
	=\lbrace u:\Vert u\Vert_{N^{1,p}(X)}<\infty\rbrace.
\end{equation*}
The Newtonian space is defined by
\[
N^{1,p}(X)=\lbrace u:\Vert u\Vert_{N^{1,p}(X)}<\infty\rbrace/\sim,
\]
where $u\sim v$ if and only if $\Vert u-v\Vert_{N^{1,p}(X)}=0$.\\

\noindent The corresponding local Newtonian space is defined by $u\in N^{1,p}_{\loc}(X)$ if
$u\in N^{1,p}(\Omega')$ for all $\Omega'\Subset X$, see \cite[Proposition 2.29]{BB},
where $\Omega'\Subset X$ means that $\overline{\Omega'}$ is a compact subset of $X$.\\

\noindent Let $\Omega$ be an open subset of $X$. We define $N^{1,p}_0(\Omega)$ as the set of functions $u\in N^{1,p}(X)$ that are zero on $X\setminus\Omega$ $\mu$-a.e. The space $N_0^{1,p}(\Omega)$ is equipped with the norm $\Vert\cdot\Vert_{N^{1,p}}$. Note also that if $\mu(X\setminus\Omega) = 0$, then $N^{1,p}_0(\Omega)=N^{1,p}(X)$. We shall therefore always assume that $\mu(X \setminus \Omega) > 0$.\\

\noindent We assume that $X$ supports the following Poincar\'{e} inequality.

\begin{definition}
Let $1\le p<\infty$. 
A metric measure space $(X,d,\mu)$ supports a weak $(1, p)$-Poincar\'{e} inequality if there exist a constant $C_{PI}$ and a dilation factor $\lambda \geq 1 $ such that 
\[
	\dashint_{B_r} |u-u_{B_r}|\, \dd\mu
	\leq C_{PI} r \left(\dashint_{B_{\lambda r}}g_u^p \, \dd\mu\right)^{\frac{1}{p}},
\]
for every ball $B_r$ in $X$ and for every $u\in L^1_{\loc}(X)$, where we denote the integral average with
\[
u_{B_r}=\dashint_{B_r} u\, \dd\mu
=\frac{1}{\mu(B_r)}\int_{B_r}u\,\dd\mu.
\]
\end{definition}

\noindent As shown in \cite[ Theorem 1.0.1 ]{KZ} by Keith and Zhong, see also \cite[Theorem 4.30]{BB}, the Poincar\'{e} inequality is a self-improving property.

\begin{theorem}\label{kz}
	Let $(X, d, \mu)$ be a complete metric measure space with a doubling measure $\mu$ and a weak $(1,p)$-Poincar\'{e} inequality with $p>1$.
	Then there exists $\varepsilon>0$ such that $X$ supports a weak $(1, s)$-Poincar\'{e} inequality for every $s>p-\varepsilon$. 
	Here, $\varepsilon$ and the constants associated with the $(1, s)$-Poincar\'e inequality depend only on $C_D$, $C_{PI}$ and $p$.
\end{theorem}
\noindent From now on and without further notice, we fix $s=s(C_{PI},C_D,p,q)$ such that $1 < s < p<q<s^*$ and for which $X$ also admits a weak $(1, s)$-Poincar\'{e} inequality. Such $s$ is given by Theorem \ref{kz} and will be used in various of our results.\\

\noindent The following result shows that the Poincar\'e inequality implies a Sobolev--Poincar\'e inequality, 
see \cite[Theorem 4.21 and Corollary 4.26]{BB}.

\begin{theorem}\label{sstars}
	Assume that $\mu$ is a doubling measure and $X$ supports a weak $(1,p)$-Poincar\'{e} inequality and let $Q=\log_2C_D$ be as in \eqref{s}.  
	Let $1\le p^*\le\frac{Qp}{Q-p}$ for $1\le p<Q$ and $1\leq p^*<\infty$ for $Q\leq p<\infty$.
	Then $X$ supports a weak $(p^*,p)$-Poincar\'{e} inequality,
	that is, there exist a constant $C=C(C_D,C_{PI},p)$ such that
	\[
		\left(\dashint_{B_r} |u-u_{B_r}|^{p^*} \,\dd\mu\right)^{\frac{1}{p^*}}
		\leq Cr\left(\dashint_{B_{2\lambda r}}g_u^p \,\dd\mu\right)^{\frac{1}{p}},
	\]
	for every ball $B_r$ in $X$ and every $u \in L^1_{\loc}(X)$. 
\end{theorem}

\begin{corollary}[\cite{BB}, Corollary 4.26]
\label{coropoinca}
If $X$ supports a weak $(1,p)$-Poincar\'{e} inequality and $Q$ in (\ref{s}) satisfies $Q\leq p$, then $X$ supports a weak $(t,p)$-Poincar\'{e} inequality for all $1\leq t<\infty$.
\end{corollary}

\begin{remark}\label{rem1poin} By the H\"older inequality we see that a weak $(p^*,p)$-Poincar\'{e} inequality implies the same inequality for smaller values of $p^*$. Meaning that $X$ will then support a weak $(t,p)$-Poincar\'{e} inequality for all $1<t<p^*$.
\end{remark}

\begin{remark}\label{rem2poin} 
The exponent $Q$ in (\ref{s}) is not uniquely determined. In particular, since $\rho < R$, one can always choose a larger value for $Q$. Therefore, the assumption $Q > p$ in Theorem \ref{sstars} can always be satisfied. 
\end{remark}

\begin{definition}[\cite{B}, Definition 6.13]
    Let $E\subset \Omega$. Then we define the variational capacity
    $$
    \textrm{cap}_p(E,\Omega)=\inf_{u}\int_{\Omega} g_u^p\dd\mu,
    $$
    where the infimum is taken over all $u\in N^{1,p}_0(\Omega)$ such that $u\geq 1$ on E.
\end{definition}
\noindent The variational capacity is also known as the relative capacity. We define the infimum to be $\infty$ in the absence of any admissible functions $u$.\\

\noindent The next result provides explicit estimates for the variational capacity, especially in the case of balls. A detailed proof is available in \cite{BB}.

\begin{proposition}\label{Prop6.16BjornBjorn} Assume that $\mu$ is a doubling measure with constant $C_D$ and $X$ supports a weak $(1,p)$-Poincar\'e inequality. Then there exists $C>0$ such that if $E\subset B_r=B(x_0,r)$ with $0<r<\frac{\textrm{diam}(X)}{8}$, then
$$
\frac{\mu(E)}{Cr^p}\leq\textrm{cap}_p(E,B_{2r})\leq\frac{C_D\mu(B_r)}{r^p}.
$$
\end{proposition}

\begin{remark}\label{REMARKcapacities} Notice that for $s_2\leq s_1$ and a small radius $r\leq 1$, we have that if $u\in N^{1,s_1}_0(B_r)$ is such that $u\geq 1$ on $E$, it is then clear that $u\in N^{1,s_2}_0(B_r)$ with $u\geq 1$ on $E$. Therefore by H\"older's inequality we obtain
\begin{align*}
    \textrm{cap}_{s_2}(E,B_r)&\leq \int_{B_r}g_u^{s_2}\,\dd\mu\\
    &=\left(\int_{B_r}g_u^{s_1}\,\dd\mu\right)^{\frac{s_2}{s_1}}\mu(B_r)^{\frac{s_1-s_2}{s_1}}.
\end{align*}
So,
$$
\mu(B_r)^{\frac{s_2-s_1}{s_1}}\textrm{cap}_{s_2}(E,B_r)\leq \left(\int_{B_r}g_u^{s_1}\,\dd\mu\right)^{\frac{s_2}{s_1}}.
$$
Therefore,
$$
\left(\mu(B_r)^{\frac{s_2-s_1}{s_1}}\textrm{cap}_{s_2}(E,B_r)\right)^{\frac{s_1}{s_2}}\leq \int_{B_r}g_u^{s_1}\,\dd\mu,\qquad\textrm{for all }u\in N^{1,s_2}_0(B_r)\textrm{ with }u\geq 1\textrm{ on }E.
$$
Meaning,
\begin{equation*}
    \mu(B_r)^{\frac{s_2-s_1}{s_1}}\textrm{cap}_{s_2}(E,B_r)\leq \textrm{cap}_{s_1}^{s_2/s_1}(E,B_r),\qquad\textrm{for }s_2\leq s_1.
\end{equation*}
Furthermore, we have
$$
\frac{1}{\textrm{cap}_{s_1}(E,B_r)}\leq\left(\mu(B_r)^{\frac{s_1-s_2}{s_1}}\frac{1}{\textrm{cap}_{s_2}(E,B_r)}\right)^{\frac{s_1}{s_2}}.
$$
Finally, since we are asking for the measure $\mu$ to be upper $Q$-Alhfors regular, by inequality \eqref{upper Q-Ahlfors} and the assumption $r\leq 1$, we conclude
\begin{equation}\label{eq2.5NEW}
\frac{1}{\textrm{cap}_{s_1}(E,B_r)}\leq\frac{C}{\textrm{cap}_{s_2}^{s_1/s_2}(E,B_r)},    
\end{equation}
where $C=C(C_1, s_1,s_2)$.
\end{remark}

\noindent The following proposition is a capacity version of the Sobolev-Poincar\'e inequality, also referred to as Maz'ya type estimate. The proof is a straightforward generalization of the Euclidean case and can be found in \cite{B}.

\begin{proposition}[\cite{B}, Proposition 3.2]\label{MazyaEstimate}Let $X$ be a doubling metric measure space supporting a weak $(1,s)$-Poincar\'e inequality. Then there exists $C$ and $\lambda\geq 1$ such that for all balls $B_r$ in $X$, $u\in N^{1,s}(X)$ and $S=\lbrace x\in B_{\frac{r}{2}}: u(x)=0\rbrace$, then
\begin{equation}
\left(\dashint_{B_r}\vert u\vert^t \, \dd\mu\right)^{\frac{1}{t}}\leq\left(\frac{C}{{\rm cap}_s(S,B_r)}\int_{B_{\lambda r}}g^s \, \dd\mu\right)^{\frac{1}{s}},
\end{equation}
for $C$ depending on $C_{PI}$ and $t$ as in Corollary \ref{coropoinca}.
\end{proposition}

\noindent Throughout the paper, positive constants are denoted by $C$ and 
the dependencies on parameters are listed in the parentheses. We denote 
\[
C(\mathrm{data})=C(C_D,C_{PI},\lambda,p,q,\alpha,[a]_\alpha).
\]

\section{Energy estimates and local boundedness}\label{local boundedness}
In this section, our main interest focuses on proving the so-called Caccioppoli inequality, Lemma \ref{DeGiorgiLemma}, and local boundedness, Corollary \ref{BoundednessInBalls}, for quasiminima of the double-phase integral \eqref{J}. We begin by defining quasiminima with boundary values.

\begin{definition}\label{qmbvJana}

Let $w\in N^{1,1}(X)$ with $H(\cdot,g_w)\in L^1(X)$.
	A function $u\in N^{1,1}(X)$  with $H(\cdot,g_u)\in L^1(X)$ is a quasiminimizer on $\Omega$ with boundary data $w$
	if $u-w\in N^{1,1}_0(\Omega)$ and there exists a constant $K\geq 1$ such that 
	\[
		\int_{\{u\ne v\}}H(x,g_u)\, \dd \mu
		\leq K\int_{\{u\ne v\}}H(x,g_v)\, \dd \mu,
	\]
for every function $v \in N^{1,1}(\Omega)$ with $u-v\in  N^{1,1}_0(\Omega)$.
\end{definition}

\begin{lemma}[Caccioppoli inequality]\label{DeGiorgiLemma} Assume that  $u\in N^{1,1}(X)$ with $H(\cdot,g_u)\in L^1(X)$  is a quasiminimizer in $\Omega$ with boundary data $w\in N^{1,1}(X)$, $H(\cdot,g_w)\in L^1(X)$. Let $x_0 \in X$, $B(x_0,r)=B_{r}$ for every $r>0$, and $S_{k,r}=\{x\in B_r:u(x)>k\}$. Then, for $0<r<R$ and $k\geq \esssup_{B_R}w$ there exists $C= C(K,q)>0$ such that the following inequality 
\begin{align}\label{CaccioppoliIneq}\int_{S_{k,r}}H(x,g_u)\, \dd \mu	& \leq C\int_{B_R} H\left(x,\frac{(u-k)_+}{R-r}\right)\, \dd\mu
\end{align} is satisfied, where $(u-k)_+= \max\{u-k,0\}$. 
\end{lemma}

\begin{proof}
The proof is a modification of the one originally given for local quasiminima of double-phase functionals in \cite[Lemma 2.9]{KNP}, which we furthermore specialize to the case of quasiminimizer $u\in N^{1,1}(X)$ with boundary data as in Definition \ref{qmbvJana}.\\

\noindent Let $\eta$ be a $(R-r)^{-1}$-Lipschitz cutoff function such that $0\leq \eta \leq1$, $\eta=1$ on $B_r$ and $\eta=0$ in $X\setminus B_R$.
Let $v=u- \eta(u-k)_+$. Then, as $u=w\leq k$ on $B_R\setminus\Omega$ and $\eta=0$ outside of $B_R$, we have $u-v\in N^{1,1}_0(\Omega)$. Also, $v=(1-\eta)(u-k)_+ +k$ on $S_{k,R}$ and $v=u$ outside $S_{k,R}$.\\

\noindent By the Leibniz rule for the upper gradients \cite[Lemma 2.18]{BB}, we have
\begin{equation*}
g_{v}\leq (u-k)_+g_{\eta}+(1-\eta) g_u +g_u \chi_{X\setminus S_{k,R}}.
\end{equation*}
Since $u$ is a quasiminimizer, by Definition \ref{qmbvJana} we obtain
\begin{equation}\label{5.7}
\begin{split}
\int_{S_{k,r}}H(x,g_u)\, \dd \mu  
&\leq\int_{\{u\neq v\}}H(x,g_u)\, \dd\mu 
\leq K\int_{\{u\neq v\}}H(x,g_v)\, \dd\mu \\
& \leq K\left(\int_{S_{k,R}} H(x, (u-k)_+g_{\eta}+(1-\eta) g_u) \, \dd \mu\right)\\
& \leq 2^qK\left(\int_{B_{R}}H\left(x,\frac{(u-k)_+}{R-r}\right)\,\dd\mu
+\int_{S_{k,R}\setminus S_{k,r}}H(x,g_u)\, \dd \mu\right).
\end{split}
\end{equation}	
By adding $K 2^{q}\int_{S_{k,r}}H(x,g_u)\, \dd \mu$ to the both sides of (\ref{5.7}), we get 
\[
(1+K 2^{q}) \int_{S_{k,r}}H(x,g_u)\, \dd \mu   
\leq K2^{q}\left(\int_{B_{R}}H\left(x,\frac{(u-k)_+}{R-r}\right)\, \dd\mu
+ \int_{S_{k,R}}H(x,g_u)\, \dd \mu\right).
\]
This implies
\begin{align*}
\int_{S_{k,r}}H(x,g_u)\, \dd \mu   
&\leq \theta\left(\int_{B_{R}}H\left(x,\frac{(u-k)_+}{R-r}\right)\, \dd\mu
+ \int_{S_{k,R}}H(x,g_u)\, \dd \mu\right)\\
&\le(R-r)^{-p}\int_{B_{R}}(u-k)_+^p\, \dd \mu
+(R-r)^{-q}\int_{B_{R}}a(u-k)_+^q\, \dd \mu
+\theta\int_{S_{k,R}}H(x,g_u)\, \dd \mu,
\end{align*}
with $\theta= \frac{K 2^{q}}{1+K 2^{q}}<1$.
We apply a standard iteration lemma, see \cite[Lemma 6.1]{G}, to obtain
\[
\int_{S_{k,r}}H(x,g_u)\, \dd \mu
\leq C\int_{B_{R}}H\left(x,\frac{(u-k)_+}{R-r}\right)\, \dd\mu,
\]
where $C=C(q,K)$.
\end{proof}

\begin{remark}\label{Proposition 3.3KS}
	Notice that if $u$ is a quasiminimizer then $-u$ is also a quasiminimizer. Therefore, by Lemma \ref{DeGiorgiLemma},  we get that $-u$ satisfies \eqref{CaccioppoliIneq}.
\end{remark}

\noindent The goal now is to prove that a quasiminimizer, as Definition \ref{qmbvJana}, is locally bounded. We first prove an auxiliary lemma. We note that a version of these results was first proved for local quasminimizers in \cite{PCN2}. The main difference is that the results proven in the present work hold for any ball in $X$ and not only for compactly contained balls in $\Omega$.

\begin{lemma}\label{lemmaestimate2} Let $u\in N^{1,1}(X)$ with $H(\cdot, g_u)\in L^1(X)$, $x_0\in X$, $B(x_0,R)=B_{R}$, $0<\frac{R}{2}<\rho<s\leq R\leq \min\lbrace 1, \frac{\diam(X)}{6}\rbrace$, and concentric balls $B_{\rho}\subset B_s\subseteq B_R$. Assume $u$ satisfies the double-phase Caccioppoli inequality \eqref{CaccioppoliIneq} for all $k\geq k^*$ and $0<r<R$. Then, for any positive numbers $k^*\leq h<k$, there exists a constant $C= C(\data, \Vert u\Vert_{N^{1,p}(X)})$, and an exponent $0<\theta= \theta(data)$, such that
   \begin{equation*}
	\int_{S_{k,\rho}}H(x,u-k) \dd\mu\leq \frac{C}{(s-\rho)^{q-p}}\left(\dfrac{\mu(S_{k, s})}{\mu(B_s)}\right)^{\theta} \int_{S_{h,s}} H\left(x,\frac{u-h}{s-\rho}\right) \dd\mu.
\end{equation*}
\end{lemma}

\begin{proof}
 Let $0<\frac{R}{2}<\rho<s\leq R\leq \min\lbrace 1, \frac{\diam(X)}{6}\rbrace$. We define $t=\frac{s+\rho}{2}$ to simplify notation. Notice that, by definition, $\rho<t<s$. By the double-phase Caccioppoli inequality \eqref{CaccioppoliIneq}, for $k\geq k^*$ we have
\begin{align}\label{ks2}
	\int_{B_t} H(x, g_{(u-k)_+}) \dd \mu&\leq  C \int_{B_s} H\left(x,\frac{(u-k)_+}{s-t}\right)\, \dd\mu \nonumber\\
& \leq  C 2^{q}  \int_{B_s} H\left(x,\frac{(u-k)_+}{s-\rho}\right)\, \dd\mu \nonumber\\
&= C \int_{B_s} H\left(x,\frac{(u-k)_+}{s-\rho}\right)\, \dd\mu,
\end{align}
where $C=C(K,q)$. Let $\tau$ be a $\dfrac{1}{s-\rho}$-Lipschitz cutoff function so that $0\leq \tau \leq1$, $\tau=1$ on $B_\rho$ and the support of $\tau$ is contained in $B_t$. Let $w=\tau (u-k)_+\in N_0^{1,1}(B_t)$. By Leibniz rule, \cite[Lemma 2.18]{BB}, we have $$g_w \leq g_{(u-k)_+}\tau +(u-k)_+ g_{\tau}\leq g_{(u-k)_+}+ \frac{1}{s-\rho}(u-k)_+,\qquad\textrm{on }B_t.$$
Using inequality \eqref{ks2}, we get
\begin{align}\label{15notes}
	\int_{B_t} H(x,g_w) \dd \mu&\leq  2^{q-1}\int_{B_t} H(x,g_{(u-k)_+}) \dd \mu+2^{q-1}\int_{B_{t}} H\left(x,\frac{(u-k)_+}{s-\rho}\right)\dd \mu \nonumber \\  
	&\leq C \, 2^{q-1} \int_{B_s} H\left(x,\frac{(u-k)_+}{s-\rho}\right)\dd \mu 
+2^{q-1} \int_{B_s}H\left(x,\frac{(u-k)_+}{s-\rho}\right) \dd \mu \nonumber \\  
	&=  C \int_{B_s}H\left(x,\frac{(u-k)_+}{s-\rho}\right) \dd \mu.
\end{align}

\noindent By H\"older inequality, the doubling property, the definition of $w$ and \cite[Lemma 4]{PCN2}, there is a constant $C=C({\rm data})$ and exponents $0<d_2<1\leq d_1<\infty$, such that

\begin{align}\label{ks40}
\dashint_{B_{\rho}}H\left(x,(u-k)_+\right)  \dd\mu&\leq \dashint_{B_{\rho}}H\left(x,\frac{(u-k)_+}{t}\right) \dd\mu\nonumber\\
&\leq   
\left(\dashint_{B_{\rho}}H\left(x,\frac{(u-k)_+}{t}\right) ^{d_1} \dd\mu\right)^{\frac{1}{d_1}}\nonumber\\ 
&=\left(\dashint_{B_{\rho}}H\left(x,\frac{\tau(u-k)_+}{t}\right) ^{d_1} \dd\mu\right)^{\frac{1}{d_1}}\nonumber\\ 
 &= \left(\dashint_{B_{\rho}}H\left(x,\frac{w}{t}\right) ^{d_1} \dd\mu\right)^{\frac{1}{d_1}}\nonumber\\ 
 &\leq \left(\dfrac{\mu\big(B_{t}\big)}{\mu(B_{\rho})}\right)^{\frac{1}{d_1}}\left(\dashint_{B_{t}} H\left(x,\frac{w}{t}\right) ^{d_1} \dd\mu\right)^{\frac{1}{d_1}} \nonumber\\	
 &\leq C\left(1+\|g_w\|^{q-p}_{L^p(B_{t})} \mu(B_{t})^{\frac{\alpha}{Q}-\frac{q-p}{p}}\right)\left(\dashint_{B_{t}}H\left(x,g_w\right) ^{d_2} \dd\mu\right)^{\frac{1}{d_2}}.
\end{align}
In the second to last inequality, we used the doubling property to estimate 
\begin{equation}\label{starextra}
    \left(\dfrac{\mu\big(B_{t}\big)}{\mu(B_{\rho})}\right)^{\frac{1}{d_1}}\leq C,
\end{equation}
where $C$ depends on the exponent $Q$ in \eqref{s}. By H\"{o}lder inequality, we have
\begin{align}\label{ks4first}
\left(\dashint_{B_{t}}H\left(x,g_w\right) ^{d_2} \dd\mu\right)^{\frac{1}{d_2}}&=\mu(B_{t})^{-\frac{1}{d_2}}\left(\int_{S_{k, t}}H\left(x,g_w\right)^{d_2} \dd\mu\right)^{\frac{1}{d_2}}\nonumber\\
&\leq \mu(B_{t})^{-\frac{1}{d_2}}\mu(S_{k, t})^{\frac{1}{d_2}-1}\int_{S_{k, t}}H\left(x,g_w\right) \dd\mu.
\end{align}
By \eqref{ks40}, \eqref{ks4first} and \eqref{15notes}, we get 
\begin{align}\label{ks4} 
&\int_{B_{\rho}}H\left(x,(u-k)_+\right)  \dd\mu\nonumber\\
&\leq C\left(1+\|g_w\|^{q-p}_{L^p(B_{t})} \mu(B_{t})^{\frac{\alpha}{Q}-\frac{q-p}{p}}\right)\left(\dfrac{\mu(S_{k, t})}{\mu(B_{t})}\right)^{\frac{1}{d_2}-1} \int_{S_{k, t}}H\left(x,g_w\right) \dd\mu \nonumber\\
 &\leq C\left(1+\|g_w\|^{q-p}_{L^p(B_t)} \mu(B_R)^{\frac{\alpha}{Q}-\frac{q-p}{p}}\right)\left(\dfrac{\mu(S_{k, t})}{\mu(B_{t})}\right)^{\frac{1}{d_2}-1}\int_{B_s}H\left(x,\frac{(u-k)_+}{s-\rho}\right) \dd \mu.
\end{align}
On the other hand, notice that 
\begin{align*}
\Vert g_w\Vert_{L^p(B_t)}^p&=\int_{B_t}g_w^p\dd\mu\nonumber\\
&\leq \int_{B_t}\left(g_{(u-k)_+}+\frac{(u-k)_+}{s-\rho}\right)^p\dd\mu\nonumber\\
&\leq 2^{p-1}\left(\int_{B_t}g_{(u-k)_+}^p\dd\mu+\int_{B_t}\left(\frac{(u-k)_+}{s-\rho}\right)^p\dd\mu\right)\\
&\leq C\left(\Vert g_u\Vert_{L^p(B_t)}^{p}+\frac{\Vert u\Vert^p_{L^p(B_t)}}{(s-\rho)^p}\right)\\
&\leq \frac{C}{(s-\rho)^p}\left(\Vert g_u\Vert_{L^p(B_t)}^{p}+\Vert u\Vert^p_{L^p(B_t)}\right).
\end{align*}
Since $u\in N^{1,1}(X)$, $H(\cdot,g_u)\in L^{1}(X)$, $B_t\subset B_R\subset X$ and the nature of the double-phase functional, then $u\in N^{1,p}(X)$. Therefore,
\begin{align*}
\Vert g_w\Vert_{L^p(B_t)}^{q-p}&\leq\frac{C}{(s-\rho)^{q-p}}\left(\Vert g_u\Vert_{L^p(B_t)}^{p}+\Vert u\Vert^p_{L^p(B_t)}\right)^{\frac{q-p}{p}}\\
&\leq\frac{C}{(s-\rho)^{q-p}}\left(\Vert g_u\Vert_{L^p(B_t)}+\Vert u\Vert_{L^p(B_t)}\right)^{q-p}\\
&=\frac{C}{(s-\rho)^{q-p}}\Vert u\Vert_{N^{1,p}(B_t)}^{q-p}\\
&\leq\frac{C}{(s-\rho)^{q-p}}\Vert u\Vert_{N^{1,p}(X)}^{q-p}.
\end{align*}
By \eqref{ks4}, the assumption that our measure is upper $Q$-Alhfors regular \eqref{upper Q-Ahlfors}, $r\leq 1$ and the last inequality, we then obtain

\begin{align}\label{new1}
\int_{B_{\rho}}&H\left(x,(u-k)_+\right)  \dd\mu\nonumber\\
&\leq \frac{C}{(s-\rho)^{q-p}}\left(1+\Vert u\Vert_{N^{1,p}(X)}^{q-p}\mu(B_R)^{\frac{\alpha}{Q}-\frac{q-p}{p}}\right)\left(\dfrac{\mu(S_{k, t})}{\mu(B_{t})}\right)^{\frac{1}{d_2}-1}\int_{B_s}H\left(x,\frac{(u-k)_+}{s-\rho}\right) \dd \mu \nonumber\\
&\leq \frac{C}{(s-\rho)^{q-p}}\left(1+\Vert u\Vert_{N^{1,p}(X)}^{q-p}\right)\left(\dfrac{\mu(S_{k, t})}{\mu(B_{t})}\right)^{\frac{1}{d_2}-1}\int_{B_s}H\left(x,\frac{(u-k)_+}{s-\rho}\right) \dd \mu  \nonumber\\
&= \frac{C}{(s-\rho)^{q-p}}\left(\dfrac{\mu(S_{k, t})}{\mu(B_{t})}\right)^{\frac{1}{d_2}-1}\int_{B_s}H\left(x,\frac{(u-k)_+}{s-\rho}\right) \dd \mu,
\end{align}
where $C=C(\data, C_1, \Vert u\Vert_{N^{1,p}(X)})$.\\

\noindent Furthermore, we observe that $\mu(S_{k, t})\leq \mu(S_{k, s})$ and $\frac{1}{d_2}-1>0$. Thus, by \eqref{new1} and the doubling property of the measure (as we did in \eqref{starextra}) we get

\begin{align}\label{algo}
\int_{B_{\rho}}H\left(x,(u-k)_+\right)  \dd\mu \leq \frac{C}{(s-\rho)^{q-p}}\left(\dfrac{\mu(S_{k, s})}{\mu(B_{s})}\right)^{\frac{1}{d_2}-1}\int_{B_s}H\left(x,\frac{(u-k)_+}{s-\rho}\right) \dd \mu.
\end{align}
with $C= C(\data, C_1, \Vert u\Vert_{N^{1,p}(X)})$.\\

\noindent Let $k^*\leq h<k$, then $(u-k)_+\leq (u-h)_+$. Therefore, we have
\begin{align*}
\int_{S_{k,\rho}}H(x,u-k) \dd\mu &\leq  \frac{C}{(s-\rho)^{q-p}}\left(\dfrac{\mu(S_{k, s})}{\mu(B_{s})}\right)^{\frac{1}{d_2}-1} \int_{S_{h,s}} H\left(x,\frac{u-h}{s-\rho}\right) \dd \mu.
\end{align*}
That is
\begin{equation*}
	\int_{S_{k,\rho}}H(x,u-k) \dd\mu\leq \frac{C}{(s-\rho)^{q-p}}\left(\dfrac{\mu(S_{k, s})}{\mu(B_s)}\right)^{\theta} \int_{S_{h,s}} H\left(x,\frac{u-h}{s-\rho}\right) \dd\mu,
\end{equation*}
where $\theta=\frac{1}{d_2}-1>0$. 
\end{proof}

\noindent As a consequence of the previous lemma, we obtain the following weak Harnack inequality type result.

\begin{theorem}\label{WeakHarnackIneqDoublePhaseBoundary}
    Let $X$ be a doubling metric measure space admitting a weak $(1,p)$-Poincar\'e inequality. Then, there exists $C>0$, $C=C(\data, C_1, R, \Vert u\Vert_{N^{1,p}(X)})$, such that if $u\in N^{1,1}(X)$ with $H(\cdot, g_u)\in L^1(X)$ and the condition \eqref{CaccioppoliIneq} holds for all $k\geq k^*$ and $0<r<R<\min\lbrace 1, \frac{\diam(X)}{6}\rbrace$, then for all $k_0\geq k^*$
    \begin{equation}\label{weakHarnackBoundary}
    \esssup_{B_{R/2}}u\leq k_0+C\left(\dashint_{B_R}H\left(x,(u-k_0)_+\right)\dd\mu\right)^{\frac{1}{p}}        
    \end{equation}
\end{theorem}

\begin{proof}
This proof is similar to the one of \cite[Theorem 3]{PCN2}, with few modifications.
Let $k_0\geq k^*$.
For $n\in \mathbb{N}\cup \{0\}$, let $\rho_{n}=\frac{R}{2}\left(1+\frac{1}{2^{n}}\right)\leq R$ and $k_{n}=k_{0}+d\left(1-\frac{1}{2^n}\right)$, where $d>0$ will be chosen later. Then, $\rho_{0}=R$, $\rho_{n}\searrow \frac{R}{2}$ and $k_n\nearrow k_0+d$.
We apply Lemma \ref{lemmaestimate2} with $\rho=\rho_{i+1}$, $R=\rho_i$, $k=k_{i+1}$ and $h=k_i$ and we get
\begin{align}\label{8notes}
 \int_{S_{k_{i+1},\rho_{i+1}}}H(x,u-k_{i+1}) \dd\mu&\leq \frac{C}{(\rho_i-\rho_{i+1})^{q-p}}\left(\dfrac{\mu(S_{k_{i+1}, \rho_i})}{\mu(B_{R})}\right)^{\theta} \int_{S_{k_i,\rho_i}} H\left(x,\frac{u-k_i}{\rho_i-\rho_{i+1}}\right) \dd\mu\nonumber \\
 & = \frac{C}{(R 2^{-i}2^{-2})^{q-p}}\left(\dfrac{\mu(S_{k_{i+1}, \rho_i})}{\mu(B_{R})}\right)^{\theta} \int_{S_{k_i,\rho_i}} H\left(x,\frac{u-k_i}{R 2^{-i}2^{-2}}\right) \dd\mu\nonumber \\
& = \frac{C2^{i(2q-p)}}{R^{q-p}}\left(\dfrac{\mu(S_{k_{i+1}, \rho_i})}{\mu(B_{R})}\right)^{\theta} \int_{S_{k_i,\rho_i}} H\left(x,\frac{u-k_i}{R}\right) \dd\mu
\nonumber\\
& \leq \frac{C4^{iq}}{R^{2q-p}}\left(\dfrac{\mu(S_{k_{i+1}, \rho_i})}{\mu(B_R)}\right)^{\theta} \int_{S_{k_i,\rho_i}} H(x,u-k_i) \dd\mu
\nonumber\\
& \leq \frac{C4^{iq}}{R^{2q}}\left(\dfrac{\mu(S_{k_{i+1}, \rho_i})}{\mu(B_R)}\right)^{\theta} \int_{S_{k_i,\rho_i}} H(x,u-k_i) \dd\mu.
\end{align}
We observe that 
\begin{align*}
    d^{-p}(k_{i+1}-k_i)^p \mu(S_{k_{i+1}, \rho_i})&=
d^{-p}\int_{S_{k_{i+1}, \rho_i}}(k_{i+1}-k_i)^p \dd\mu\\
&\leq d^{-p}\int_{S_{k_{i+1}, \rho_i}}H(x, u-k_i) \dd\mu\\
&\leq d^{-p}\int_{S_{k_i, \rho_i}}H(x, u-k_i) \dd\mu.
\end{align*}
So, we have 
\begin{equation*}
    \psi_i= d^{-p}\int_{S_{k_i, \rho_i}}H(x, u-k_i) \dd\mu \geq d^{-p}(k_{i+1}-k_i)^p \mu(S_{k_{i+1}, \rho_i}). 
\end{equation*}
This implies
\begin{equation}\label{9notes}
 \mu(S_{k_{i+1}, \rho_i})\leq \psi_i d^{p}(k_{i+1}-k_i)^{-p}. 
\end{equation}
Therefore, by \eqref{8notes} and \eqref{9notes} we obtain 
\begin{align*}
   d^p \psi_{i+1}&\leq \frac{C4^{iq}}{R^{2q}} \left(\dfrac{\mu(S_{k_{i+1}, \rho_i})}{\mu(B_R)}\right)^{\theta} d^p\psi_i\nonumber\\ 
    &\leq \frac{C4^{iq}}{R^{2q}} \left(\psi_i d^{p}(k_{i+1}-k_i)^{-p}\right)^{\theta}d^p\psi_i \mu(B_R)^{-\theta}\nonumber\\
     &= \frac{C4^{iq}}{R^{2q}} (d2^{-1-i})^{-p\theta} d^{p\theta}\psi_i^{1+\theta} d^p \mu(B_R)^{-\theta}\nonumber\\
    &\leq  \frac{C4^{(1+\theta)qi}}{R^{2q}} \psi_i^{1+\theta}d^p \mu(B_R)^{-\theta}.
\end{align*}
That is,
\begin{equation}
 \psi_{i+1} \leq\frac{C4^{(1+\theta)qi}}{R^{2q}} \psi_i^{1+\theta}\mu(B_R)^{-\theta},
\end{equation}
for every $i\geq0$, where $C=C({\rm data},C_1, \Vert u\Vert_{N^{1,p}(X)})$.
By using a standard iteration lemma \cite[Lemma 7.1]{G}
we get \begin{equation*}
 \lim_{i\to \infty}\psi_i=\lim_{i\to \infty}d^{-p}\int_{S_{k_i, \rho_i}}H(x, u-k_i) \dd\mu=0,   
\end{equation*} provided that $d=CR^{\frac{-2q}{p\theta}} \left(\dashint_{B_R}H(x, (u-k_0)_+) \dd\mu\right)^{\frac{1}{p}}>0$. As a consequence, \begin{equation*}
    \int_{B_{\frac{R}{2}}} H(x, (u-(k_0+d))_+ \dd\mu=0
\end{equation*}
and so
\begin{equation}\label{quasiboundedness}
    u\leq k_0+d \quad\mbox{almost everywhere in $B_{\frac{R}{2}}$ and for all $k_0
    \geq k^*$}.
\end{equation}
We conclude that
\begin{align*}
   \esssup_{B_{\frac{R}{2}}}u &\leq k_0 +d= k_0+ C \left(\dashint_{B_R}H\left(x,(u-k_0)_+\right) \, \dd\mu\right)^{\frac{1}{p}},
\end{align*}
where $C=C(\data, C_1, R, \Vert u\Vert_{N^{1,p}(X)})$.
\end{proof}

\noindent As a corollary of Theorem \ref{WeakHarnackIneqDoublePhaseBoundary}, we obtain local boundedness for quasiminima of \eqref{J}. We emphasize that these results are a generalization of those initially obtained in \cite{PCN2}, since here the conclusions are valid for any ball contained in $X$ and not only for balls compactly contained in the domain $\Omega$.

\begin{corollary}\label{BoundednessInBalls} Let $u\in N^{1,1}(X)$ with $H(\cdot, g_u)\in L^1(X)$. Assume $u$ satisfies the double-phase Caccioppoli inequality \eqref{CaccioppoliIneq} for all $k\geq k^*$ and $0<r<R<\min\lbrace 1, \frac{\diam(X)}{6}\rbrace$. Then, 
$$
\Vert u\Vert_{L^{\infty}(B_{R/2})}<\infty
$$
holds. In particular, if $u$ is a quasiminimizer in $\Omega$ with boundary data $w\in N^{1,1}(X)$, $H(\cdot,g_w)\in L^1(X)$, then $u$ is locally bounded. 
\end{corollary}
\begin{proof} Since $u$ satisfies the double-phase Caccioppoli inequality \eqref{CaccioppoliIneq}, then by Theorem \ref{WeakHarnackIneqDoublePhaseBoundary} and, in particular, by inequality \eqref{quasiboundedness}, we obtain the desiered result.\\

\noindent In case $u$ is a quasiminimizer, by Lemma \ref{DeGiorgiLemma}, $u$ satisfies the double-phase Caccioppoli inequality \eqref{CaccioppoliIneq} for all $k\geq \esssup_{B(x_0,R)}w$. Without loss of generality, we can assume $\esssup_{B(x_0,R)}w<\infty$, otherwise by inequality \eqref{quasiboundedness} the result trivially holds true.   
    Therefore, the result follows as before.   
\end{proof}

\noindent Once we have that quasiminima are locally bounded in any ball $B_r\subset X$, we can prove the following almost standard Caccioppoli's inequality. This states that, in the $p$-regime case, double-phase quasiminima satisfy a Caccioppoli type inequality analogous to the one satisfied for quasiminima of functionals with just $p$-growth plus some extra controllable terms. This result was proved in \cite{PCN2} for local quasiminima of double-phase problems. In the Euclidean case this result has been proven by Colombo-Mingione \cite{CM}.
\begin{lemma}[Almost standard Caccioppoli's inequality]\label{Almost standard Caccioppoli's inequality}
Assume that  $u\in N^{1,1}(X)$ with $H(\cdot,g_u)\in L^1(X)$  satisfies the double-phase Caccioppoli inequality \eqref{CaccioppoliIneq} for all $k\geq k^*$ and $0<r<R<\min\lbrace 1, \frac{\diam(X)}{6}\rbrace$. Furthermore, assume that
  \begin{equation}\label{p-regime}
      \sup_{B(x_0,R)} a(x) \leq C [a]_{\alpha} \mu(B(x_0,R))^{\frac{\alpha}{Q}},
  \end{equation}
 holds for $x_0\in X$, $0\leq R\leq 1$.
  Then there exists $C= C(C_1, p, q,\alpha, Q ,[a]_{\alpha}, \|u\|_{L^{\infty}(B(x_0,R))}, K)>0$ such that for any choice of concentric balls $B(x_0,s)\subset B(x_0,t) \subset B(x_0,R)\subset X $, with $0<t<s\leq R\leq 1$ and $2\Vert u\Vert_{L^{\infty}(B(x_0,R))}\geq \vert k\vert\geq k\geq\esssup_{B(x_0,R)}w$, the following inequality 
\begin{align}\label{ASCaccioppoli}
\int_{B(x_0,t)}g_{(u-k)_+}^p\, \dd \mu	& \leq C \left(\left(\frac{R}{s-t}\right)^q\int_{B(x_0,s)} \left|\frac{(u-k)_+}{R}\right|^p\, \dd\mu\right)
\end{align}
is satisfied, where $(u-k)_+= \max\{u-k,0\}$. In particular, if $u$ is a quasiminimizer in $\Omega$ with boundary data $w\in N^{1,1}(X)$, $H(\cdot, g_w)\in L^1(X)$, then \eqref{ASCaccioppoli} holds as well.
\end{lemma}
\begin{proof}
For this proof, we treat the case when $u$ is a quasiminimizer. By Lemma \ref{DeGiorgiLemma}, since $u$ is a quasiminimizer with boundary values $w$, then for all $k\geq\esssup_{B(x_0,R)}w$, $u$ satisfies the double-phase Caccioppoli's inequality \eqref{CaccioppoliIneq}. Therefore, we have 
 \begin{align}\label{asc1}
\int_{B(x_0,t)}g_{(u-k)_+}^p\, \dd \mu &\leq   \int_{B(x_0,t)}H(x,g_{(u-k)_+})\, \dd \mu    \leq  C \int_{B(x_0,s)} H\left(x,\frac{(u-k)_+}{s-t}\right)\, \dd\mu \nonumber\\
& = C \left(\int_{B(x_0,s)} \left|\frac{(u-k)_+}{s-t}\right|^p\, \dd\mu + \int_{B(x_0,s)} a(x)\left|\frac{(u-k)_+}{s-t}\right|^q\, \dd\mu\right)\nonumber\\
& =  C \Bigg(\left(\frac{R}{s-t}\right)^p\int_{B(x_0,s)} \left|\frac{(u-k)_+}{R}\right|^p\, \dd\mu \nonumber\\
&\qquad\qquad\qquad+\left(\frac{R}{s-t}\right)^q \int_{B(x_0,s)} a(x)\left|\frac{(u-k)_+}{R}\right|^q\, \dd\mu\Bigg),
 \end{align}
 here $C=C(K,q)$.\\
 
\noindent Now, we estimate the integrand function in the last term of the previous inequality. Using \eqref{p-regime} and  \eqref{upper Q-Ahlfors}, we get
\begin{align*}
a(x)\left|\frac{(u-k)_+}{R}\right|^q &\leq C[a]_{\alpha}\frac{\mu(B(x_0,R))^{\frac{\alpha}{Q}}}{R^q} \|u\|^{q-p}_{L^{\infty}(B(x_0,s))} |(u-k)_+|^p\\
& \leq C[a]_{\alpha}\frac{C_1 R^{\alpha}}{R^q} \|u\|^{q-p}_{L^{\infty}(B(x_0,R))} |(u-k)_+|^p\\
&= \frac{C[a]_{\alpha}}{R^{q-\alpha}} \|u\|^{q-p}_{L^{\infty}(B(x_0,R))} |(u-k)_+|^p\\
&\leq \frac{C [a]_{\alpha}}{R^{p}}\|u\|^{q-p}_{L^{\infty}(B(x_0,R))} |(u-k)_+|^p= C \left|\frac{(u-k)_+}{R}\right|^p,
\end{align*}
 where $C=C(C_1, p, q,\alpha, Q ,[a]_{\alpha}, \|u\|_{L^{\infty}(B(x_0,R))})$. The last inequality holds true by \eqref{pqcond} and Remark \ref{rem2poin}.
 Thus, \eqref{asc1} becomes
\begin{align*}
\int_{B(x_0,t)}g_{(u-k)_+}^p\, \dd \mu &\leq  C \left(\left(\frac{R}{s-t}\right)^p\int_{B(x_0,s)} \left|\frac{(u-k)_+}{R}\right|^p\, \dd\mu +C\left(\frac{R}{s-t}\right)^q\int_{B(x_0,s)} \left|\frac{(u-k)_+}{R}\right|^p\, \dd\mu\right)\\
& \leq C \left(\frac{R}{s-t}\right)^q \int_{B(x_0,s)} \left|\frac{(u-k)_+}{R}\right|^p\, \dd\mu.
 \end{align*}
\end{proof}

\section{Frozen functionals}\label{frozen}
In this section, we collect the necessary regularity results for quasiminima of the so-called frozen functionals used in subsequent sections. We consider functionals of the type

\begin{equation}\label{eq1frozen}
\int_{\Omega}H_0(g_u)\dd\mu=\int_{\Omega}(g_u^p+a_0g_u^q)\dd\mu,
\end{equation}
where $a_0\geq 0$ is a constant. Frozen functionals belong to the class of functionals introduced by the seminal work of Lieberman \cite{Lieberman}. One of the main difficulties when doing the careful analysis of the different phases of the double-phase functional can be solved by using frozen functionals theory, which is based on the fact that under appropriate conditions a quasiminimizer $u$ of functional \eqref{J} can be seen as a quasiminimizer of a certain frozen functional.\\

\noindent The related theory for this notion due to Lieberman \cite{Lieberman} can be used to obtain the needed estimates for regularity properties. For example, local regularity results were obtained in the Euclidean case in \cite{CM} and in the general context of a metric measure space in \cite{PCN2}. Here we prove the following generalization of Maz'ya's estimate, Proposition \ref{MazyaEstimate} (see also \cite[Theorem 6.21]{BB}), for frozen functionals. As far as we know this result is new even in the Euclidean case. 

\begin{theorem}[Maz'ya's type estimate for frozen functionals and small radii]\label{Mazya_Frozen} Let $u\in N^{1,1}(X)$ with $H_0(g_u)\in L^1(X)$, $B_r\subset X$ any ball with  radius $r\leq \min\lbrace 1, \frac{\diam(X)}{8}\rbrace$ and $S=\lbrace x\in B_{\frac{r}{2}}: u(x)=0\rbrace$. Then, there exists  $C=C(C_D,C_{\textrm{PI}}, C_1,p,q)>0$ and exponents $d_1\geq 1>d_2>0$, with $d_1=d_1(C_D,C_{\textrm{PI}},p,q)$ and $d_2=d_2(C_D,C_{\textrm{PI}},p, q)$, such that the following inequality holds
$$
\left(\dashint_{B_r} H_0\left(u\right)^{d_1}\dd \mu\right)^{\frac{1}{d_1}}\leq C\left(\frac{1}{\textrm{cap}_{pd_2}^{q/p}(S,B_r)}\int_{B_{2\lambda r}} H_0(g_u)^{d_2}\dd \mu\right)^{\frac{1}{d_2}},
$$
where $\lambda$ is the dilation constant in the $(1,p)$-Poincar\'e inequality.
\end{theorem}

\begin{proof}
By splitting $u$ into its positive and negative parts and considering them separately, we can assume that $u\geq 0$ in ${B_r}$. 
Recall that we are assuming that our space supports a $(1,s)$-Poincaré inequality with $1<s<p<q<s^*$. By Remark \ref{rem2poin} and Theorem \ref{sstars}, $X$ supports a $(s^*,s)$-Poincar\'e inequality.\\

 \noindent Let $\frac{s}{p}<d_2<1$, $d_2=d_2(C_{\textrm{PI}}, C_D,p,q)$, and $\frac{s^*}{q}\geq d_1\geq 1$, $d_1=d_1(C_{\textrm{PI}}, C_D,p,q)$. By Proposition \ref{MazyaEstimate}, the following two inequalities hold at the same time
\begin{equation}\label{eq2.3NEW}
    \left(\dashint_{B_r}\left\vert u\right\vert^{p d_1}\dd\mu\right)^{\frac{d_2}{d_1}}\leq \frac{C}{\textrm{cap}_{p d_2}(S, B_r)}\int_{B_{2\lambda r}}g_u^{p d_2}\dd\mu
\end{equation}
and
\begin{equation}\label{eq2.4NEW}
    \left(\dashint_{B_r}\left\vert u\right\vert^{q d_1}\dd\mu\right)^{\frac{d_2}{d_1}}\leq \frac{C}{\textrm{cap}_{q d_2}(S, B_r)}\int_{B_{2\lambda r}}g_u^{q d_2}\dd\mu
\end{equation}

\noindent Therefore, by \eqref{eq2.4NEW} and \eqref{eq2.5NEW} in Remark \ref{REMARKcapacities}, we have
\begin{align}\label{eq2.6NEW}
    \left(\dashint_{B_r}\left\vert u\right\vert^{q d_1}\dd\mu\right)^{\frac{d_2}{d_1}}&\leq \frac{C}{\textrm{cap}_{q d_2}(S, B_r)}\int_{B_{2\lambda r}}g_u^{q d_2}\dd\mu \nonumber\\
   &\leq \frac{C}{\textrm{cap}_{p d_2}^{q/p}(S, B_r)}\int_{B_{2\lambda r}}g_u^{q d_2}\dd\mu,
\end{align}
where $C=C(C_{\textrm{PI}}, C_1,p,q)$. Now, by \eqref{eq2.3NEW}, \eqref{eq2.6NEW}, Proposition \ref{Prop6.16BjornBjorn}, $\mu$ being upper $Q$-Alhfors regular and $r\leq 1$, we obtain the next chain of inequalities

\begin{align*}
    \Bigg(\dashint_{B_r}H_0^{d_1}\left(u\right)\dd\mu\Bigg)^{\frac{1}{d_1}}&\leq \left(\dashint_{B_r}2^{d_1-1}\left(\left\vert u\right\vert^{pd_1}+a_0^{d_1}\left\vert u\right\vert^{qd_1}\right)\dd\mu\right)^{\frac{1}{d_1}}\\
    &\leq 2\left(\left(\dashint_{B_r}\left\vert u\right\vert^{pd_1}\dd\mu\right)^{\frac{1}{d_1}}+a_0\left(\dashint_{B_r}\left\vert u\right\vert^{qd_1}\dd\mu\right)^{\frac{1}{d_1}}\right)\\
    &\leq C\Bigg( \left(\frac{1}{\textrm{cap}_{p d_2}(S,B_r)}\int_{B_{2\lambda r}}g_u^{pd_2}\dd\mu\right)^{\frac{1}{d_2}}+\left(\frac{1}{\textrm{cap}_{p d_2}^{q/p}(S,B_r)}\int_{B_{2\lambda r}}(a_0g_u^q)^{d_2}\dd\mu\right)^{\frac{1}{d_2}}\Bigg)\\
    &\leq C\Bigg( \left(\left(1+\frac{1}{\textrm{cap}_{p d_2}(S,B_r)}\right)\int_{B_{2\lambda r}}g_u^{pd_2}\dd\mu\right)^{\frac{1}{d_2}}\\
    &\qquad \qquad+\left(\left(1+\frac{1}{\textrm{cap}_{p d_2}(S,B_r)}\right)^{q/p}\int_{B_{2\lambda r}}(a_0g_u^q)^{d_2}\dd\mu\right)^{\frac{1}{d_2}}\Bigg)\\
    &\leq C\left(1+\frac{1}{\textrm{cap}_{p d_2}(S,B_r)}\right)^{\frac{q}{pd_2}}\left(\left(\int_{B_{2\lambda r}}g_u^{pd_2}\dd\mu\right)^{\frac{1}{d_2}}+\left(\int_{B_{2\lambda r}}(a_0g_u^q)^{d_2}\dd\mu\right)^{\frac{1}{d_2}}\right)\\
    &\leq C\left(\left(\frac{\textrm{cap}_{p d_2}(S,B_r)+1}{\textrm{cap}_{p d_2}(S,B_r)}\right)^{\frac{q}{pd_2}}\left(\int_{B_{2\lambda r}}g_u^{pd_2}\dd\mu+\int_{B_{2\lambda r}}(a_0g_u^q)^{d_2}\dd\mu\right)^{\frac{1}{d_2}}\right)\\
    &\leq C\left(\left(\frac{\left(\frac{C_D\mu(B_{\frac{r}{2}})}{\left(\frac{r}{2}\right)^{pd_2}}\right)+1}{\textrm{cap}_{p d_2}(S,B_r)}\right)^{q/p}\int_{B_{2\lambda r}}\left(g_u^{pd_2}+(a_0g_u^q)^{d_2}\right)\dd\mu\right)^{\frac{1}{d_2}}\\
     &\leq C\left(\left(\frac{C_DC_1\left(\frac{r}{2}\right)^{Q-pd_2}+1}{\textrm{cap}_{p d_2}(S,B_r)}\right)^{q/p}\int_{B_{2\lambda r}}\left(g_u^{pd_2}+(a_0g_u^q)^{d_2}\right)\dd\mu\right)^{\frac{1}{d_2}}\\
    &\leq C\left(\left(\frac{C+1}{\textrm{cap}_{p d_2}(S,B_r)}\right)^{q/p}\int_{B_{2\lambda r}}\left(g_u^{pd_2}+(a_0g_u^q)^{d_2}\right)\dd\mu\right)^{\frac{1}{d_2}}\\
    &\leq C\left(\frac{1}{\textrm{cap}_{pd_2}^{q/p}(S,B_R)}\int_{B_{2\lambda r}}H_0(g_u)^{d_2}\dd\mu\right)^{\frac{1}{d_2}}.
\end{align*}
Therefore,
$$
 \left(\dashint_{B_r}H_0^{d_1}\left(u\right)\dd\mu\right)^{\frac{1}{d_1}}\leq C\left(\frac{1}{\textrm{cap}_{pd_2}^{q/p}(S,B_R)}\int_{B_{2\lambda r}}H_0(g_u)^{d_2}\dd\mu\right)^{\frac{1}{d_2}},
$$
where $C=C(C_{\textrm{PI}}, C_D, C_1,p,q)$. 
\end{proof}

\section{Pointwise estimate on a boundary point}\label{Pointwise estimate on a boundary point}
This section is devoted to the proof of a pointwise estimate near a boundary point which has a key role in obtaining the sufficient condition for H\"older continuity (Theorem \ref{Theorem2.11}, Section \ref{Sufficient condition for Holder continuity}). \\

\noindent Let $w\in N^{1,1}(X)$ with $H(\cdot, g_w)\in L^1(X)$, and such that  $w - u \in N^{1,1}_0(\Omega)$.
We shall use the notation 
\begin{equation*}
	M(r,r_{0})=\Big(\esssup_{B(x_{0},r)} u -\esssup_{B(x_{0},r_{0})} w\Big)_{+},
\end{equation*}
where $0<r\leq r_{0}$, $a_{+}=\max\lbrace a,0\rbrace$, and $u \in N^{1,1}(X)$, with $H(\cdot, g_u)\in L^1(X)$ is a quasiminimizer on $\Omega$, with boundary values $w$.
Let also
\begin{equation}
	\gamma_{p,q}(s,r)=\frac{r^{-sq/p}\mu(B(x_{0},r))}{{\rm cap}_{s}^{q/p}(B(x_{0},r)\setminus\Omega, B(x_{0},2r))}.
\end{equation}

\noindent Now we state the main result of this section, a pointwise estimate for quasiminima near a boundary point. The proof of this estimate is based on a careful analysis of the phases.

\begin{proposition}[Pointwise estimate]\label{pointwise-estimate} Let $w\in N^{1,1}(X)$ with $H(\cdot, g_w)\in L^1(X)$. Let $u \in N^{1,1}(X)$, with $H(\cdot, g_u)\in L^1(X)$ a quasiminimizer on $\Omega$ with boundary values $w$. Then there exist $\lambda\geq 1$, $0<d_2<1$ and $C>0$ such that, for all $x_{0}\in\partial\Omega$ and $0<4\lambda r\leq r_{0}<\min\lbrace 1, \frac{\textrm{diam}(X)}{12 \lambda}\rbrace$, the next inequality holds
	\begin{equation*}
		M\left(\frac{r}{2}, r_{0}\right)\leq (1-2^{-n(r)-1})M(4\lambda r, r_{0}),
	\end{equation*}
	where $n(r)$ is a sufficiently large integer such that
	\begin{equation*}
		n(r)\geq C\gamma_{p,q}\left(pd_2,\frac{r}{2}\right)^{\frac{1}{1-d_2}}.
	\end{equation*}
 
\end{proposition}

\begin{proof}
We denote $M=M(4\lambda r, r_0)$, where $x_0\in\partial\Omega$ and $r_0>0$ are fixed. Without loss of generality, we can assume $0<M<+\infty$, otherwise the proof is finished. 

\noindent We define
$$
k_j=\esssup_{B(x_0,r_0)}w+M(1-2^{-j}),
$$
and
$$
v_j=(u-k_j)_+-(u-k_{j+1})_+.
$$
Notice that in $B(x_0,2\lambda r)\setminus\Omega$ we have $v_j=0$, where $\lambda$ is given by Theorem \ref{Mazya_Frozen}. Define $T(k,l,r)=S_{k,r}\setminus S_{l,r}$, then $g_uX_{T(k_j,k_{j+1},2\lambda r)}$ is a $p$-weak upper gradient of $v_j$ in $B(x_0,2\lambda r)$.\\

\noindent We define
$$
a_0=\inf_{x\in B(x_0, 8\lambda r)}a(x).
$$

\noindent \textbf{Case 1:} First assume that 
\begin{equation}\label{case1INFIMUM}
    a_0>C[a]_{\alpha}\mu(B(x_0,8\lambda r))^{\frac{\alpha}{Q}},
\end{equation}
Note that for every $x,y\in  B(x_0, 8\lambda r)$, we have
\begin{align*}
\delta_{\mu}(x,y)
&=\left(\mu(B(x,d(x,y)))+\mu(B(y,d(x,y)))\right)^{1/Q}\\
&\leq \left(\mu(B(x,16\lambda r))+\mu(B(y,16\lambda r))\right)^{1/Q}\\
&\leq \left(2\mu (B(x_0,24\lambda r))\right)^{1/Q}
\leq (2C_D^2\mu(B(x_0, 8\lambda r)))^{1/Q}\\
&= C\mu(B(x_0, 8\lambda r))^{1/Q}.
\end{align*}
By \eqref{case1INFIMUM} we obtain
\begin{align*}
2a_0&= 2a(x)-2\left(a(x)-a_0\right)\geq a(x)+a_0-2\left(a(x)-a_0\right)\\ 
&\geq a(x)+2[a]_{\alpha}(2C_D^2\mu(B(x_0, 8\lambda r)))^{\alpha/Q}-2\left(a(x)-a_0\right)\\
&\geq a(x)+2\sup_{\substack{x,y \in B(x_0, 8\lambda r)\\ x\neq y}} \frac{|a(x)-a(y)|}{\delta_{\mu}(x,y)^{\alpha}}(2C_D^2\mu(B(x_0, 8\lambda r)))^{\alpha/Q}-2\left(a(x)-a_0\right)\\
&\geq a(x)+2\sup_{x,y \in B(x_0, 8\lambda r)}|a(x)-a(y)|-2\left(a(x)-a_0\right)\\
&\geq a(x)+2\sup_{x,y \in B(x_0, 8\lambda r)}(a(x)-a(y))-2\left(a(x)-a_0\right)\\
&\geq a(x)+2a(x)-2\inf_{y \in B(x_0, 8\lambda r)}a(y)-2\left(a(x)-a_0\right)
=a(x),
\end{align*}
for every $x \in B(x_0, 8\lambda r)$.
On the other hand, we have $a(x)\geq\inf_{x\in B(x_0, 8\lambda r)} a(x) \geq a_0$ for every $x\in B(x_0, 8\lambda r)$.
This implies that 
\begin{equation}\label{eq2pointwiseProof}
a_0\le a(x)\le 2a_0\qquad \textrm{for every }x \in B(x_0, 8\lambda r).
\end{equation}
We define the Frozen functional
$$
H_0(x)=\vert z\vert^p+a_0\vert z\vert^q.
$$

\noindent By Theorem \ref{Mazya_Frozen} applied in $B(x_0,8\lambda r)$ and with $v_j$, there exist $d_1\geq 1> d_2>0$, with $d_1=d_1(C_{\textrm{PI}},C_D,p,q)$ and $d_2=d_2(C_{\textrm{PI}},C_D,p,q)$, such that by H\"older inequality we have
\begin{align*}
    \dashint_{B(x_0,r)}H_0\left(v_j\right)\dd\mu&\leq \left(\dashint_{B(x_0,r)}H_0\left(v_j\right)^{d_1}\dd\mu\right)^{\frac{1}{d_1}}\\
    &\leq C\left(\frac{1}{\textrm{cap}_{pd_2}^{q/p}(B\left(x_0,\frac{r}{2}\right)\setminus\Omega, B(x_0,r))}\int_{B(x_0, 2\lambda r)}H_0(g_{v_j})^{d_2}\dd\mu\right)^{\frac{1}{d_2}}\\
    &= C \left(\frac{r^{qd_2}\gamma_{p,q}\left(pd_2,\frac{r}{2}\right)}{\mu\left(B\left(x_0,\frac{r}{2}\right)\right)}\int_{B(x_0,2\lambda r)}H_0(g_{v_j})^{d_2}\dd\mu\right)^{\frac{1}{d_2}}\\
     &\leq C \left(\frac{r^{qd_2}\gamma_{p,q}\left(pd_2,\frac{r}{2}\right)}{\mu\left(B\left(x_0,\frac{r}{2}\right)\right)}\int_{T(k_j,k_{j+1},2\lambda r)}H_0(g_{u})^{d_2}\dd\mu\right)^{\frac{1}{d_2}}\\
     &=C\left(\frac{r^{qd_2}\gamma_{p,q}\left(pd_2,\frac{r}{2}\right)\mu(T(k_j,k_{j+1},2\lambda r))}{\mu\left(B\left(x_0,\frac{r}{2}\right)\right)}\right)^{\frac{1}{d_2}}\left(\dashint_{T(k_j,k_{j+1},2\lambda r)}H_0(g_u)^{d_2}\dd\mu\right)^{\frac{1}{d_2}}\\
     &\leq \frac{Cr^q\gamma_{p,q}\left(pd_2,\frac{r}{2}\right)^{\frac{1}{d_2}}\mu(T(k_j,k_{j+1},2\lambda r)^{\frac{1}{d_2}-1}}{\mu\left(B\left(x_0,\frac{r}{2}\right)\right)^{\frac{1}{d_2}}}\int_{T(k_j,k_{j+1},2\lambda r)}H_0(g_u)\dd\mu\\
     &\leq C\left(\frac{r^{qd_2}\gamma_{p,q}\left(pd_2,\frac{r}{2}\right)\mu(T(k_j,k_{j+1},2\lambda r)^{1-d_2}}{\mu\left(B\left(x_0,\frac{r}{2}\right)\right)}\right)^{\frac{1}{d_2}}\int_{S_{k_j,2\lambda r}}H_0(g_u)\dd\mu.
\end{align*}
Therefore, by the doubling property we obtain
\begin{equation}\label{1Prop2}
      \int_{B(x_0,r)}H_0\left(v_j\right)\dd\mu\leq C\left(\frac{\mu\left(B\left(x_0,\frac{r}{2}\right)\right)}{\mu(T(k_j,k_{j+1},2\lambda r)}\right)^{1-\frac{1}{d_2}}\left(r^{qd_2}\gamma_{p,q}\left(pd_2,\frac{r}{2}\right)\right)^{\frac{1}{d_2}}\int_{S_{k_j,2\lambda r}}H_0(g_u)\dd\mu.
\end{equation}
For the measure of the set $S_{k_{j+1},r}$ we have the following estimate
\begin{align*}
    \int_{B(x_0,r)}H_0\left(v_j\right)\dd\mu&\geq  H_0\left(k_{j+1}-k_j\right)\mu(S_{k_{j+1},r})\\
    &= H_0\left(2^{-j-1}M\right)\mu(S_{k_{j+1},r})\\
    &=H_0\left(\frac{M}{2^{j+1}}\right)\mu(S_{k_{j+1},r})
\end{align*}
Therefore,
\begin{equation}\label{2Prop2}
     \int_{B(x_0,r)}H_0\left(v_j\right)\dd\mu\geq H_0\left(\frac{M}{2^{j+1}}\right)\mu(S_{k_{j+1},r})
\end{equation}
Now, by \eqref{eq2pointwiseProof}, Lemma \ref{DeGiorgiLemma}, because the functional $H_0$ is increasing and $r\leq 1$, we achieve
\begin{align*}
    \int_{S_{k_j,2\lambda r}} H_0(g_u)\dd\mu&\leq \int_{S_{k_j,2\lambda r}} H(x,g_u)\dd\mu\leq \int_{B(x_0,2\lambda r)}H(x,g_{(u-k_j)_+})\dd\mu\\
     &\leq C \int_{B(x_0,4\lambda r)}H\left(x,\frac{(u-k_j)_+}{2\lambda r}\right)\dd\mu\leq 2C \int_{B(x_0,4\lambda r)}H_0\left(\frac{(u-k_j)_+}{ r}\right)\dd\mu\\
    &\leq C \int_{B(x_0,4\lambda r)}H_0\left(\frac{\left(\esssup_{B(x_0,4\lambda r)}u-k_j\right)_+}{ r}\right)\dd\mu\\
    &\leq C\int_{B(x_0,4\lambda r)}H_0\left(\frac{M}{r2^j}\right)\dd\mu= C \mu(B(x_0,4\lambda r))H_0\left(\frac{M}{r2^j}\right)\\
    &= C\mu(B(x_0,4\lambda r))\left(\left\vert\frac{M}{r2^j}\right\vert^p+a_0\left\vert\frac{M}{r2^j}\right\vert^q\right)\\
    &\leq  C\mu(B(x_0,4\lambda r))\frac{1}{r^q}H_0\left(\frac{M}{2^j}\right) .
\end{align*}
So,
\begin{equation}\label{3Prop2}
      \int_{S_{k_j,2\lambda r}} H_0(g_u)\dd\mu\leq C\mu(B(x_0,4\lambda r))\frac{1}{r^q}H_0\left(\frac{M}{2^j}\right),
\end{equation}
where $C=C(K,q)$.

\noindent By \eqref{2Prop2}, \eqref{1Prop2} and \eqref{3Prop2}, and the doubling property of the measure, we have
\begin{align*}
    H_0\left(\frac{M}{2^{j+1}}\right)\mu(S_{j+1,r})&\leq \int_{B(x_0,r)}H_0\left(v_j\right)\dd\mu\\
    &\leq C\left(\frac{\mu\left(B\left(x_0,\frac{r}{2}\right)\right)}{\mu(T(k_j,k_{j+1},2\lambda r)}\right)^{1-\frac{1}{d_2}}\left(r^{qd_2}\gamma_{p,q}\left(pd_2,\frac{r}{2}\right)\right)^{\frac{1}{d_2}}\int_{S_{k_j,2\lambda r}}H_0(g_u)\dd\mu\\
    &\leq  C\left(\frac{\mu\left(B\left(x_0,\frac{r}{2}\right)\right)}{\mu(T(k_j,k_{j+1},2\lambda r)}\right)^{1-\frac{1}{d_2}}r^{q}\gamma_{p,q}\left(pd_2,\frac{r}{2}\right)^{\frac{1}{d_2}}\mu(B(x_0,4\lambda r))\frac{1}{r^q}H_0\left(\frac{M}{2^j}\right)\\
    &\leq C\left(\frac{\mu\left(B\left(x_0,\frac{r}{2}\right)\right)}{\mu(T(k_j,k_{j+1},2\lambda r)}\right)^{1-\frac{1}{d_2}}\gamma_{p,q}\left(pd_2,\frac{r}{2}\right)^{\frac{1}{d_2}}\mu(B(x_0, r))H_0\left(\frac{M}{2^j}\right).
\end{align*}
By inequality (5.9) in \cite{PCN2}, we get
$$
\frac{1}{2^qq}H_0\left(\frac{M}{2^j}\right)\leq H_0\left(\frac{M}{2^{j+1}}\right).
$$
So, 
$$
\frac{1}{2^qq}H_0\left(\frac{M}{2^j}\right)\mu(S_{j+1,r})\leq C\left(\frac{\mu\left(B\left(x_0,\frac{r}{2}\right)\right)}{\mu(T(k_j,k_{j+1},2\lambda r)}\right)^{1-\frac{1}{d_2}}\gamma_{p,q}\left(pd_2,\frac{r}{2}\right)^{\frac{1}{d_2}}\mu(B(x_0, r))H_0\left(\frac{M}{2^j}\right).
$$
Therefore,
\begin{equation}\label{4Prop2}
  \frac{\mu(S_{k_{j+1},r})}{\mu(B(x_0,r))}\leq C\left(\frac{\mu(T(k_j,k_{j+1},2\lambda r)}{\mu\left(B\left(x_0,\frac{r}{2}\right)\right)}\right)^{\frac{1}{d_2}-1}  \gamma_{p,q}\left(pd_2,\frac{r}{2}\right)^{\frac{1}{d_2}},
\end{equation}
where $C=C(C_{\textrm{PI}}, C_D,C
_1,K,p,q)$ and $0<d_2<1$, $d_2=d_2(C_{\textrm{PI}},C_D,p,q)$.\\ 

\noindent If $n\geq j+1$ then $S_{k_{j+1},r}$ on the left-hand side of \eqref{4Prop2} can be replaced by $S_{k_n,r}$ and the inequality remains true. We get

$$
\left(\frac{\mu(S_{k_{n},r})}{\mu(B(x_0,r))}\right)^{\frac{d_2}{1-d_2}}\leq C\gamma_{p,q}\left(pd_2,\frac{r}{2}\right)^{\frac{1}{1-d_2}}\frac{\mu(T(k_j,k_{j+1},2\lambda r)}{\mu\left(B\left(x_0,\frac{r}{2}\right)\right)},
$$
summing up over $j=0,1,\cdots n-1$ and Lemma \ref{ineqDoubling}, yields
\begin{align*}
    n\left(\frac{\mu(S_{k_{n},r})}{\mu(B(x_0,r))}\right)^{\frac{d_2}{1-d_2}}&\leq C\gamma_{p,q}\left(pd_2,\frac{r}{2}\right)^{\frac{1}{1-d_2}}\frac{\mu(T(k_1,k_{n},2\lambda r)}{\mu\left(B\left(x_0,\frac{r}{2}\right)\right)}\\
    &\leq C\gamma_{p,q}\left(pd_2,\frac{r}{2}\right)^{\frac{1}{1-d_2}}\frac{\mu(B(x_0,2\lambda r))}{\mu\left(B\left(x_0,r\right)\right)}\\
    &\leq C\gamma_{p,q}\left(pd_2,\frac{r}{2}\right)^{\frac{1}{1-d_2}}\left(\frac{2\lambda r}{r}\right)^Q\\
    &\leq C\gamma_{p,q}\left(pd_2,\frac{r}{2}\right)^{\frac{1}{1-d_2}}.
\end{align*}

\noindent Therefore,

\begin{equation}\label{5Prop2}
    \frac{\mu(S_{k_{n},r})}{\mu(B(x_0,r))}\leq\frac{C}{n^{\frac{1}{d_2}-1}}\gamma_{p,q}\left(pd_2,\frac{r}{2}\right)^{\frac{1}{d_2}},
\end{equation}
with $C=C(C_{\textrm{PI}}, C_D, C_1, K,\lambda, p,q)$ and $0<d_2<1$, $d_2=d_2(C_{\textrm{PI}},C_D,p,q)$. \\

\noindent Since $u$ is a quasiminimizer with boundary data $w$ on $\Omega$, the Caccioppoli inequality
$$
\int_{S_{k,\rho_1}}H(x,g_u)\dd\mu\leq C\int_{B(x_0,\rho_2)}H\left(x,\frac{(u-k)_+}{\rho_2-\rho_1}\right)\dd\mu,
$$
holds for all $0<\rho_1<\rho_2$ and $k\geq \esssup_{B(x_0,\rho_2)}w$. In particular, by \eqref{eq2pointwiseProof}, for any $0<\rho_1<\rho_2\leq 8\lambda r$ and $k\geq\esssup_{B(x_0,8\lambda r)}w$, we have

\begin{align*}
  \int_{B(x_0,\rho_1)}H_0(g_{(u-k)_+})\dd\mu&\leq \int_{S_{k,\rho_1}}H(x,g_u)\dd\mu\\
&\leq C\int_{B(x_0,\rho_2)}H\left(x,\frac{(u-k)_+}{\rho_2-\rho_1}\right)\dd\mu\\
&\leq 2C \int_{B(x_0,\rho_2)}H_0\left(\frac{(u-k)_+}{\rho_2-\rho_1}\right)\dd\mu.  
\end{align*}
Furthermore, notice that $(u-k_n)$ also satisfies the previous inequality in the ball $B(x_0, 8\lambda r)$. Therefore, by an application of \cite[Theorem 6]{PCN2} with $\Omega=B(x_0,8\lambda r)$, $u=(u-k_n)_+$ and $t_+=p$, there exist $C=C(C_{\textrm{PI}}, C_D, p,q)$ such that 

\begin{align*}
    \esssup_{B\left(x_0,\frac{r}{2}\right)}(u-k_n)&\leq \esssup_{B\left(x_0,\frac{r}{2}\right)}(u-k_n)_+\\
    &\leq\frac{C}{\left(1-\frac{r/2}{r}\right)^{\frac{Q}{p}}}\left(\dashint_{B(x_0,r)}(u-k_n)_+^p\dd\mu\right)^{\frac{1}{p}}\\
&\leq C\left(\dashint_{B(x_0,r)}(u-k_n)_+^p\dd\mu\right)^{\frac{1}{p}}\\
&=C\left(\frac{1}{\mu(B(x_0,r))}\int_{S_{k_n,r}}(u-k_n)^p\dd\mu\right)^{\frac{1}{p}}\\
&\leq C\left(\frac{\mu(S_{k_n,r})}{\mu(B(x_0,r))}\right)^{\frac{1}{p}}\left(\esssup_{B(x_0,r)}u-k_n\right).
\end{align*}
Therefore, by \eqref{5Prop2}
\begin{align*}
 \esssup_{B\left(x_0,\frac{r}{2}\right)}u&\leq k_n+C\left(\frac{\mu(S_{k_n,r})}{\mu(B(x_0,r))}\right)^{\frac{1}{p}}\left(\esssup_{B(x_0,r)}u-k_n\right)\\
 &=\esssup_{B(x_0,r_0)}w+M(1-2^{-n})+C\left(\frac{\mu(S_{k_n,r})}{\mu(B(x_0,r))}\right)^{\frac{1}{p}}\left(\esssup_{B(x_0,r)}u-\esssup_{B(x_0,r_0)}w-M(1-2^{-n})\right)\\
 &\leq \esssup_{B(x_0,r_0)}w+M(1-2^{-n})+C\left(\frac{\mu(S_{k_n,r})}{\mu(B(x_0,r))}\right)^{\frac{1}{p}}M2^{-n}\\
 &\leq\esssup_{B(x_0,r_0)}w+M(1-2^{-n})+C\left(\frac{1}{n^{\frac{1}{d_2}-1}}\gamma_{p,q}\left(pd_2,\frac{r}{2}\right)^{\frac{1}{d_2}}\right)^{\frac{1}{p}}M2^{-n},
\end{align*}
with $C=C(C_{\textrm{PI}}, C_D,C_1, K, \lambda, p,q)$.\\

\noindent If $n\geq \left(2^{pd_2}\gamma_{p,q}\left(pd_2,\frac{r}{2}\right)\right)^{\frac{1}{1-d_2}}$, then $\left(\frac{1}{n^{\frac{1}{d_2}-1}}\gamma_{p,q}\left(pd_2,\frac{r}{2}\right)^{\frac{1}{d_2}}\right)^{\frac{1}{p}}\leq\frac{1}{2}$. Therefore, the last term on the right-hand side of the previous inequality is at most $2^{-n-1}M$. \\

\noindent Thus, for 

\begin{equation}\label{condition_n}
n\geq C\gamma_{p,q}\left(pd_2,\frac{r}{2}\right)^{\frac{1}{1-d_2}},
\end{equation} we have
\begin{align*}
     \esssup_{B\left(x_0,\frac{r}{2}\right)}u&\leq \esssup_{B(x_0,r_0)}w+M-2^{-n}M+2^{-n-1}M\\
     &= \esssup_{B(x_0,r_0)}w +M(1-2^{-n-1}).
\end{align*}

\noindent Finally, we get
\begin{equation}\label{conclusion1}
M\left(\frac{r}{2}, r\right)\leq (1-2^{-n-1})M,
\end{equation}
for $n$ satisfying \eqref{condition_n}. \textit{End Case 1.}\\

\noindent \textbf{Case 2:} Next we consider the case which is complementary to \eqref{case1INFIMUM}, that is, 
\begin{equation}\label{3C}
a_0=\inf_{x\in B(x_0, 8\lambda r)}a(x)\leq C[a]_{\alpha}\mu(B(x_0, 8\lambda r))^{\alpha/Q}.
\end{equation}
Notice that, for every $x\in  B(x_0, 8\lambda r)$ and $y \in B(x_0,  r)$, with $y\neq x$, we have
\begin{align}\label{notice}
    a(y)-a(x)&\leq \vert a(x)-a(y)\vert=\frac{\vert a(x)-a(y)\vert}{\delta_{\mu}(x,y)^{\alpha}}\delta_{\mu}(x,y)^{\alpha}\leq  [a]_{\alpha}\delta_{\mu}(x,y)^{\alpha}.
\end{align}
Note that, for every $x\in  B(x_0, 8\lambda r)$ and $y \in B(x_0,  r)$, with $y\neq x$, we have
\begin{align*}
\delta_{\mu}(x,y)
&=\left(\mu(B(x,d(x,y)))+\mu(B(y,d(x,y)))\right)^{1/Q}\\
&\leq \left(\mu(B(x,9\lambda r))+\mu(B(y,9\lambda r))\right)^{1/Q}\\
&\leq \left(2\mu (B(x_0,10\lambda r))\right)^{1/Q}
\leq C\mu(B(x_0, 8\lambda r))^{1/Q},
\end{align*}
where $C=C(C_D)$.
By \eqref{notice}, we get
\[
a(y)\leq a(x)+C[a]_{\alpha}\mu(B(x_0, 8\lambda r))^{\alpha/Q},
\]
where $C=C(C_D,\alpha)$.
By taking infimum over all $x\in 8\lambda B(x_0,  r)$, we obtain
\begin{align*}
 a(y)
 &\leq \inf_{x\in B(x_0, 8\lambda r)}a(x)+C[a]_{\alpha}\mu(B(x_0, 8\lambda r))^{\alpha/Q}\\
 &\leq 2[a]_{\alpha}C(\mu(B(x_0, 8\lambda r)))^{\alpha/Q}+C[a]_{\alpha}\mu(B(x_0, 8\lambda r))^{\alpha/Q}\\\
 &\leq C[a]_{\alpha}\mu(B(x_0, 8\lambda r))^{\alpha/Q},
\end{align*}
where $C=C(C_D,\alpha)$.
By taking supremum over $y\in B(x_0,  r)$, we conclude that
\begin{equation}\label{eq11POINTWISE}
\sup_{y\in B(x_0,  r)}a(y)
\leq C[a]_{\alpha}\mu(B(x_0, 8\lambda r))^{\alpha/Q}.
\end{equation}
Since $1<s<pd_2<p$, as a consequence of Theorem \ref{kz} and Remark \ref{rem1poin}, $X$ supports a weak $(pd_2,pd_2)$-Poincar\'e inequality. Therefore, by the weak $(pd_2,pd_2)$-Poincar\'e inequality, Proposition \ref{Prop6.16BjornBjorn}, the upper $Q$-Ahlfors inequality \eqref{upper Q-Ahlfors}, $r\leq 1$, Remark \ref{rem2poin}, the doubling property of $\mu$ and H\"older inequality, we get 
\begin{align*}
    \int_{B(x_0,r)}v_j^{pd_2}\dd\mu&\leq \frac{C\mu(B(x_0,r))}{\textrm{cap}_{pd_2}(B(x_0,\frac{r}{2})\setminus\Omega, B(x_0,r))}\int_{B(x_0,2\lambda r)}g_{v_j}^{pd_2}\dd\mu\\
    &\leq C\mu(B(x_0,r))\left(1+\frac{1}{\textrm{cap}_{pd_2}(B(x_0,\frac{r}{2})\setminus\Omega, B(x_0,r))}\right)\int_{B(x_0,2\lambda r)}g_{v_j}^{pd_2}\dd\mu\\
    &\leq C\mu(B(x_0,r))\left(1+\frac{1}{\textrm{cap}_{pd_2}(B(x_0,\frac{r}{2})\setminus\Omega, B(x_0,r))}\right)^{\frac{q}{p}}\int_{B(x_0,2\lambda r)}g_{v_j}^{pd_2}\dd\mu\\
    &=C\mu(B(x_0,r))\left(\frac{\textrm{cap}_{pd_2}(B(x_0,\frac{r}{2})\setminus\Omega, B(x_0,r))+1}{\textrm{cap}_{pd_2}(B(x_0,\frac{r}{2})\setminus\Omega, B(x_0,r))}\right)^{\frac{q}{p}}\int_{B(x_0,2\lambda r)}g_{v_j}^{pd_2}\dd\mu\\        
    &\leq C\mu(B(x_0,r))\left(\frac{\left(\frac{C_D\mu(B(x_0,\frac{r}{2}))}{\left(\frac{r}{2}\right)^{pd_2}}\right)+1}{\textrm{cap}_{pd_2}(B(x_0,\frac{r}{2})\setminus\Omega, B(x_0,r))}\right)^{\frac{q}{p}}\int_{B(x_0,2\lambda r)}g_{v_j}^{pd_2}\dd\mu\\  
    &\leq C\mu(B(x_0,r))\left(\frac{C_DC_1\left(\frac{r}{2}\right)^{Q-pd_2}+1}{\textrm{cap}_{pd_2}(B(x_0,\frac{r}{2})\setminus\Omega, B(x_0,r))}\right)^{\frac{q}{p}}\int_{B(x_0,2\lambda r)}g_{v_j}^{pd_2}\dd\mu\\  
    &\leq C\mu(B(x_0,r))\left(\frac{C+1}{\textrm{cap}_{pd_2}(B(x_0,\frac{r}{2})\setminus\Omega, B(x_0,r))}\right)^{\frac{q}{p}}\int_{B(x_0,2\lambda r)}g_{v_j}^{pd_2}\dd\mu\\  
    &=\frac{ C\mu(B(x_0,r))}{\textrm{cap}_{pd_2}^{q/p}(B(x_0,\frac{r}{2})\setminus\Omega, B(x_0,r))}\int_{B(x_0,2\lambda r)}g_{v_j}^{pd_2}\dd\mu\\  
    &\leq C\gamma_{p,q}(pd_2,\frac{r}{2})r^{qd_2}\int_{S_{k_j,2\lambda r}}g_u^{pd_2}\chi_{T(k_j,k_{j+1},2\lambda r)}\dd\mu\\
     &\leq C\gamma_{p,q}(pd_2,\frac{r}{2})r^{pd_2}\left(\int_{S_{k_j,2\lambda r}}g_u^{p}\dd\mu\right)^{d_2}\mu(T(k_j,k_{j+1},2\lambda r))^{1-d_2},
\end{align*}
with $C$ depending on $C_{\textrm{PI}}$, $C_D$ and $C_1$, therefore
\begin{equation}\label{eq12POINTWISE}
     \int_{B(x_0,r)}v_j^{pd_2}\dd\mu\leq C\gamma_{p,q}(pd_2,\frac{r}{2})r^{pd_2}\left(\int_{S_{k_j,2\lambda r}}g_u^{p}\dd\mu\right)^{d_2}\mu(T(k_j,k_{j+1},2\lambda r))^{1-d_2}.
\end{equation}
For the measure of the set $S_{k_{j+1},r}$ we have the following estimate
\begin{equation}\label{eq13POINTWISE}
 \int_{B(x_0,r)}v_j^{pd_2}\dd\mu\geq(k_{j+1}-k_j)^{pd_2}\mu(S_{k_{j+1}, r})=\frac{M^{pd_2}}{2^{pd_2(j+1)}}\mu(S_{k_{j+1},r}) .   
\end{equation}
Since $u$ is a quasiminimizer with boundary data $w$ in $\Omega$, by \eqref{eq11POINTWISE}, Lemma \ref{Almost standard Caccioppoli's inequality}, H\"older inequality and the doubling property, we have for $R=8\lambda r$, $s=2\lambda r$, $t=\lambda r$
\begin{align*}
    \int_{S_{k_j,\lambda r}}g_u^p\dd\mu&=\int_{B(x_0,\lambda r)}g_{(u-k_j)_+}^p\dd\mu\leq C\left(\left(\frac{8\lambda r}{\lambda r}\right)\int_{B(x_0,2\lambda r)}\left\vert\frac{(u-k_j)_+}{8\lambda r}\right\vert^p\dd\mu\right)\\
    &\leq \frac{C}{r^p}\int_{B(x_0,2\lambda r)}(u-k_j)_+^p\dd\mu\leq \frac{C}{r^p}\int_{B(x_0,2\lambda r)}\left(\esssup_{B(x_0,2\lambda r)}u-k_j\right)_+^p\dd\mu\\
    &\leq C\left(\frac{M}{r2^j}\right)^p\mu(B(x_0,2\lambda r)).
\end{align*}
Therefore,
\begin{equation}\label{eq14POINTWISE}
\int_{S_{k_j,\lambda r}}g_u^p\dd\mu\leq C\left(\frac{M}{r2^j}\right)^p\mu(B(x_0,2\lambda r)).    
\end{equation}
By \eqref{eq13POINTWISE}, \eqref{eq12POINTWISE} and \eqref{eq14POINTWISE}, we obtain
\begin{align*}
    \mu(S_{k_{j+1},r})\frac{M^{pd_2}}{2^{pd_2(j+1)}}&\leq \int_{B(x_0,r)}v_j^{pd_2}\dd\mu\\
    &\leq C\gamma_{p,q}(pd_2,\frac{r}{2})r^{pd_2}\left(\int_{S_{k_j,2\lambda r}}g_u^{pd_2}\dd\mu\right)^{d_2}\mu(T(k_j,k_{j+1},2\lambda r))^{1-d_2}\\
    &\leq  C\gamma_{p,q}(pd_2,\frac{r}{2})r^{pd_2}\frac{CM^{pd_2}}{r^{pd_2}2^{jpd_2}}\mu(B(x_0,2\lambda r))^{d_2}\mu(T(k_j,k_{j+1},2\lambda r))^{1-d_2}\\
    &=  C\gamma_{p,q}(pd_2,\frac{r}{2})\frac{CM^{pd_2}}{2^{jpd_2}}\mu(B(x_0,2\lambda r))^{d_2}\mu(T(k_j,k_{j+1},2\lambda r))^{1-d_2}.
\end{align*}
So,
\begin{equation}\label{eq15POINTWISE}
  \frac{\mu(S_{k_{j+1},r})}{\mu(B(x_0,r))}\leq C\gamma_{p,q}(pd_2,\frac{r}{2})\left(\frac{\mu(T(k_j,k_{j+1},2\lambda r))}{\mu(B(x_0,r))}\right)^{1-d_2}.
\end{equation}
As before, if $n\geq j+1$, then \eqref{eq15POINTWISE} holds true when the set $S_{k_{j+1},r}$ on the left part of the inequality is replaced by $S_{k_n,r}$. Thus, by Lemma \ref{ineqDoubling} 

\begin{align*}
    n\left( \frac{\mu(S_{k_{n},r})}{\mu(B(x_0,r))}\right)^{\frac{1}{1-d_2}}&\leq  C\gamma_{p,q}(pd_2,\frac{r}{2})^{\frac{1}{1-d_2}}\left(\frac{\mu(T(k_j,k_{j+1},2\lambda r))}{\mu(B(x_0,r))}\right)\\
    &\leq C\gamma_{p,q}(pd_2,\frac{r}{2})^{\frac{1}{1-d_2}}\left(\frac{\mu(B(x_0,2\lambda r))}{\mu(B(x_0,r))}\right)\\
    &\leq  C\gamma_{p,q}(pd_2,\frac{r}{2})^{\frac{1}{1-d_2}}\left(\frac{2\lambda r}{r}\right)^{Q}.
\end{align*}
Meaning,
\begin{equation}\label{eq16POINTWISE}
    \frac{\mu(S_{k_{n},r})}{\mu(B(x_0,r))}\leq\frac{C}{n^{1-d_2}}\gamma_{p,q}(pd_2,\frac{r}{2}).
\end{equation}
Again, since $u$ is a quasiminimizer with boundary data $w$ in $\Omega$, then by Lemma \ref{DeGiorgiLemma}, in particular, $u$ satisfies the Caccioppoli inequality in the ball $B(x_0,8\lambda r)$. Since \eqref{eq11POINTWISE} holds, by an application of \cite[Theorem 11]{PCN2} with $\Omega=B(x_0,8\lambda r)$, $u=(u-k_n)_+$ and $t_+=p$, there exists $C=C(\data, C_1, \Vert u\Vert_{L^\infty(B(x_0,r))})$ such that

\begin{align*}
    \esssup_{B\left(x_0,\frac{r}{2}\right)}(u-k_n)&\leq \esssup_{B\left(x_0,\frac{r}{2}\right)}(u-k_n)_+\\
    &\leq\frac{C}{\left(1-\frac{r/2}{r}\right)^{\frac{Q}{p}}}\left(\dashint_{B(x_0,r)}(u-k_n)_+^p\dd\mu\right)^{\frac{1}{p}}\\
&\leq C\left(\dashint_{B(x_0,r)}(u-k_n)_+^p\dd\mu\right)^{\frac{1}{p}}\\
&=C\left(\frac{1}{\mu(B(x_0,r))}\int_{S_{k_n,r}}(u-k_n)^p\dd\mu\right)^{\frac{1}{p}}\\
&\leq C\left(\frac{\mu(S_{k_n,r})}{\mu(B(x_0,r))}\right)^{\frac{1}{p}}\left(\esssup_{B(x_0,r)}u-k_n\right).
\end{align*}

\noindent Therefore, by \eqref{eq16POINTWISE} 
\begin{align*}
 \esssup_{B\left(x_0,\frac{r}{2}\right)}u&\leq k_n+C\left(\frac{\mu(S_{k_n,r})}{\mu(B(x_0,r))}\right)^{\frac{1}{p}}\left(\esssup_{B(x_0,r)}u-k_n\right)\\
 &=\esssup_{B(x_0,r_0)}w+M(1-2^{-n})+C\left(\frac{\mu(S_{k_n,r})}{\mu(B(x_0,r))}\right)^{\frac{1}{p}}\left(\esssup_{B(x_0,r)}u-\esssup_{B(x_0,r_0)}w-M(1-2^{-n})\right)\\
 &\leq \esssup_{B(x_0,r_0)}w+M(1-2^{-n})+C\left(\frac{\mu(S_{k_n,r})}{\mu(B(x_0,r))}\right)^{\frac{1}{p}}M2^{-n}\\
 &\leq\esssup_{B(x_0,r_0)}w+M(1-2^{-n})+C\left(\frac{1}{n^{1-d_2}}\gamma_{p,q}\left(pd_2,\frac{r}{2}\right)\right)^{\frac{1}{p}}M2^{-n},
\end{align*}

\noindent If $n\geq \left(2^p\gamma_{p,q}\left(pd_2,\frac{r}{2}\right)\right)^{\frac{1}{1-d_2}}$, then $\left(\frac{1}{n^{1-d_2}}\gamma_{p,q}\left(pd_2,\frac{r}{2}\right)\right)^{\frac{1}{p}}\leq\frac{1}{2}$. Therefore, the last term on the right-hand side of the previous inequality is at most $2^{-n-1}M$. \\

\noindent Thus, for 

\begin{equation}\label{condition_n2}
n\geq C\gamma_{p,q}\left(pd_2,\frac{r}{2}\right)^{\frac{1}{1-d_2}},
\end{equation} we have
\begin{align*}
     \esssup_{B\left(x_0,\frac{r}{2}\right)}u&\leq \esssup_{B(x_0,r_0)}w+M-2^{-n}M+2^{-n-1}M\\
     &= \esssup_{B(x_0,r_0)}w +M(1-2^{-n-1}).
\end{align*}

\noindent Finally, we get
\begin{equation}\label{conclusion2}
M\left(\frac{r}{2}, r\right)\leq (1-2^{-n-1})M,
\end{equation}
for $n$ satisfying \eqref{condition_n2}. \textit{End Case 2.}\\

\noindent Notice that equations \eqref{condition_n} and \eqref{condition_n2} are the same, and we can uniform conditions \eqref{conclusion1} and \eqref{conclusion2}.\\

\noindent Therefore, combining both cases, we obtain that by choosing the smallest integer such that  $n\geq C\gamma_{p,q}\left(pd_2,\frac{r}{2}\right)^{\frac{1}{1-d_2}}$,  we finish the proof.
\end{proof}

\section{Sufficient condition for H\"older continuity}\label{Sufficient condition for Holder continuity}
This last section contains the main result of the present work, Theorem \ref{Theorem2.12}. It is worth mentioning that after obtaining the pointwise estimate on a boundary point, see Proposition \ref{pointwise-estimate}, the arguments used in this section are inspired by techniques from the literature, particularly \cite{B, NP}, but adapted to our setting. For completeness, we provide full proofs below. We start with the following Theorem, a pointwise estimate for quasiminima near a boundary point, see \cite{B, Maz, NP}. 
\begin{theorem}\label{Theorem2.11}
	Let $w\in N^{1,1}(X)$ with $H(x,g_w)\in L^1(X)$. Let $u \in N^{1,1}(X)$, with $H(x,g_u)\in L^1(X)$ a quasiminimizer on $\Omega$ with boundary values $w$.  Then there exist $C_{0},C_{1}>0$ such that
	\begin{equation}
		M(\rho,r_{0})\leq C_{1}M(r_{0},r_{0})\exp\left(-\frac{1}{4}\int_{\rho}^{r_{0}}\exp\left(-C_{0}\gamma_{p,q}(pd_2,r)^{\frac{1}{1-d_2}}\right)\frac{\, d r}{r}\right).
	\end{equation}
 \end{theorem}

\begin{proof}
	Let $M(r)=M(r,r_{0})$, where $r_{0}>0$ is fixed. It is not restrictive to suppose that $0< M(r_{0}) < +\infty$.
	We consider $C$ and $n(r)$ as in Proposition \ref{pointwise-estimate}, $C_{0}=C\log 2$ and
	\begin{equation*}
		\zeta(r)=\exp\left(-C_{0}\gamma_{p,q}(pd_2,r)^{\frac{1}{1-d_2}}\right)=2^{-n(2r)}.
	\end{equation*}
	We partition $(0,r_{0})=I_{1}\cup I_{2}$ with $I_{1}\cap I_{2}=\emptyset$, where
	\begin{equation*}
		I_{m}=\bigcup_{i=1}^{+\infty}\left[(8\lambda)^{m-2i-1}r_{0}, (8\lambda)^{m-2i}r_{0}\right), \quad \mbox{for $m=1,2$.} 
	\end{equation*}
	It follows, 
	\begin{equation}\label{eq1theo2.11}
		\int_{\rho}^{r_{0}}\zeta(r)\frac{dr}{r}\leq 2\int_{(\rho,r_{0})\cap I_{m}}\zeta(r)\frac{dr}{r},
	\end{equation} for $m=1$ or $m=2$. For each $i \in \mathbb{N}$, we define $r_{i}\in \left[(8\lambda)^{m-2i-1}r_{0}, (8\lambda)^{m-2i}r_{0}\right)$  satisfying
	\begin{equation*}
		\zeta(r_{i})\geq\frac{1}{(8\lambda)^{m-2i-1}r_{0}}\int_{(8\lambda)^{m-2i-1}r_{0}}^{(8\lambda)^{m-2i}r_{0}}\zeta(r)dr\geq \int_{(8\lambda)^{m-2i-1}r_{0}}^{(8\lambda)^{m-2i}r_{0}}\zeta(r)\frac{dr}{r}.
	\end{equation*}
	Since $\zeta(r)\leq 1$ for all $r$, we obtain
	\begin{equation}\label{eq2theo2.11}
		\int_{(\rho,r_{0})\cap I_{m}}\zeta(r)\frac{dr}{r}\leq\sum_{\rho\leq r_{i}<\frac{r_{0}}{8\lambda}}\zeta(r_{i})+C.
	\end{equation}
	By Proposition \ref{pointwise-estimate}, we get
	\begin{align*}
		M((8\lambda)^{m-2i-1}r_{0})&\leq M(r_{i})\leq M(8\lambda r_{i})(1-2^{-n(2r_{i})-1})\\
		&\leq M((8\lambda)^{m-2i+1}r_{0})\left(1-\frac{\zeta(r_{i})}{2}\right), \quad \mbox{for $i\geq 1$.}
	\end{align*}
	By an iteration process and the trivial inequality $1-t\leq \exp(-t)$, we deduce
	\begin{equation*}
		M(\rho)\leq M(r_{0})\exp\left(-\frac{1}{2}\sum_{\rho\leq r_{i}<\frac{r_{0}}{8\lambda}}\zeta(r_{i})\right),  \quad \mbox{for $0<\rho<r_{0}$.}
	\end{equation*}
	Furthermore,
	\begin{equation*}
		-\frac{1}{2}\sum_{\rho\leq r_{i}<\frac{r_{0}}{8\lambda}}\zeta(r_{i})\leq C-\frac{1}{2}\int_{(\rho,r_{0})\cap I_{m}}\zeta(r)\frac{dr}{r}\leq C-\frac{1}{4}\int_{\rho}^{r_{0}}\zeta(r)\frac{dr}{r}.
	\end{equation*}
	Therefore,
	\begin{equation*}
		\exp\left(-\frac{1}{2}\sum_{\rho\leq r_{i}<\frac{r_{0}}{8\lambda}}\zeta(r_{i})\right)\leq C_{1}\exp\left(-\frac{1}{4}\int_{\rho}^{r_{0}}\exp\left(-C_{0}\gamma_{p,q}(pd_2,r)^{\frac{1}{1-d_2}}\right)\frac{dr}{r}\right).
	\end{equation*}
\end{proof}
\noindent Finally, Theorem  \ref{Theorem2.11} implies the sufficient condition for H\"older continuity of quasiminima of double-phase functionals at a boundary point. 

\begin{theorem}\label{Theorem2.12} Let $w\in N^{1,1}(X)$ with $H(x,g_w)\in L^1(X)$. Let $u \in N^{1,1}(X)$, with $H(x,g_u)\in L^1(X)$ a quasiminimizer on $\Omega$ with boundary values $w$. 
	If $w$ is a H\"older continuous function at $x_{0}\in\partial\Omega$, then there exists $C_0>0$ such that
	\begin{equation*}
		\liminf_{\rho\rightarrow 0}\frac{1}{\vert \log \rho\vert}\int_{\rho}^{1}\exp\left(-C_0 \gamma_{p,q}(pd_2,r)^{\frac{1}{1-d_2}}\right)\frac{dr}{r}>0.
	\end{equation*}
	Thus $u$ is H\"older continuous at $x_{0}$.
\end{theorem}

\begin{proof} We notice that $-u$ is a quasiminimizer (see Remark \ref{Proposition 3.3KS}). Furthermore,  by hypothesis $-u-(-w) \in N^{1,1}_{0}(\Omega)$, so it is enough to work with the positive part $(u(x)-w(x_{0}))_{+}$. Without loss of generality, we can make $w(x_{0})=0$. Being $w$ continuous at $x_{0}$, Corollary \ref{BoundednessInBalls} yields to
	\begin{equation}\label{eq1theo2.12}
		M(r_{0},r_{0})\leq M=\esssup_{B(x_{0},R)}u_{+}<+\infty,
	\end{equation} for some $R>0$ and all $0<r_{0}<R$.
	
	\noindent For $0<\rho<r_{0}$, using Theorem \ref{Theorem2.11} we obtain 
	\begin{align}\label{eq2theo2.12}
		\nonumber \esssup_{B(x_{0},\rho)}u_{+}&\leq \esssup_{B(x_{0},r_{0})}w_{+}+M(\rho,r_{0})\\
		&\leq \esssup_{B(x_{0},r_{0})}w_{+}+C_{1}M\exp\left(-\frac{1}{4}\int_{\rho}^{r_{0}}\exp\left(-C_{0}\gamma_{p,q}(pd_2,r)^{\frac{1}{1-d_2}}\right)\frac{dr}{r}\right).
	\end{align} 
So there are $C, h,k>0$ satisfying $\esssup_{B(x_{0},r_{0})}w_{+}\leq Cr_{0}^{h}$
	and, for all sufficiently small $\rho$ and $r_{0}$,
	\begin{equation*}
		\int_{\rho}^{1}\exp\left(-C_{0}\gamma_{p,q}(pd_2,r)^{\frac{1}{1-d_2}}\right)\frac{dr}{r}\geq k\vert\log\rho\vert.
	\end{equation*}

	\noindent We observe that
	\begin{equation*}
		\int_{r_{0}}^{1}\exp\left(-C_{0}\gamma_{p,q}(pd_2,r)^{\frac{1}{1-d_2}}\right)\frac{dr}{r}\leq \int_{r_{0}}^{1}\frac{dr}{r}=\vert\log r_{0}\vert, 
	\end{equation*}for all $0<r_{0}<1$. For sufficiently small $\rho$ and $r_{0}$, using inequality (\ref{eq2theo2.12}), we derive 
	\begin{equation*}
		\esssup_{B(x_{0},\rho)}u_{+}\leq Cr_{0}^{h}+C_{1}M\rho^{\frac{k}{4}}r_{0}^{-\frac{1}{4}}.
	\end{equation*}

	\noindent The proof of the H\"older continuity of $u$ at $x_0$ is completed by choosing $r_{0}=\rho^{k '}$ with $0<k '<k$.
    \end{proof}

\section*{Acknowledgements}

\noindent A. Nastasi is a member of the \textit{Gruppo Nazionale per
l'Analisi Matematica, la Probabilit\`{a} e le loro Applicazioni} (GNAMPA) of
the Istituto Nazionale di Alta Matematica (INdAM) and was partly
supported by GNAMPA-INdAM Project  \textit{Regolarit\`{a} per problemi
ellittici e parabolici con crescite non standard}, CUP E53C22001930001.\\
\noindent This research was partly conducted while C. Pacchiano Camacho was at the
Okinawa Institute of Science and Technology (OIST) through the Theoretical
Sciences Visiting Program (TSVP).  C. Pacchiano Camacho was supported by a grant from Simons Foundation International SFI-MPS-T-Institutes-00011977 JS. \newline



\begin{thebibliography}{10}


\bibitem{BCM} \textsc{P. Baroni, M. Colombo, G. Mingione}, \textit{Harnack inequalities for double-phase functionals}, Nonlinear Anal., 121 (2015), 206--222.



\bibitem{B} \textsc{J. Bj\"{o}rn}, \textit{Boundary continuity for quasiminimizers on metric spaces}, Illinois J. Math., 46 (2002), 383--403.
	 
	 
\bibitem{BB}  \textsc{A. Bj\"{o}rn, J. Bj\"{o}rn}, \textit{Nonlinear potential theory on metric spaces}, EMS Tracts in Mathematics, 17, European Mathematical Society (EMS), Zurich (2011). 
	
\bibitem{BBS} \textsc{A. Bj\"{o}rn, J. Bj\"{o}rn, N. Shanmugalingam}, \textit{The Dirichlet problem for $p$-harmonic functions on metric spaces}, J. Reine Angew. Math., 556 (2003), 173--203.

\bibitem{BMS} \textsc{J.  Bj\"{o}rn, P. MacManus, N. Shanmugalingam},  \textit{Fat sets and pointwise boundary estimates for $p$-harmonic functions in metric spaces}, J. Anal. Math., 85 (2001), 339--369.


\bibitem{CKKSS} \textsc{L. Capogna, J. Kline, R. Korte, N. Shanmugalingam, M. Snipes}, \textit{Neumann problems for $p$-harmonic functions, and induced nonlocal operators in metric measure spaces}, to appear in Amer. J. Math. arxiv: \url{https://arxiv.org/abs/2204.00571}.

\bibitem{C} \textsc{J. Cheeger}, \textit{Differentiability of Lipschitz functions on metric spaces}, Geom. Funct. Anal., 9 (1999), 428--517.



\bibitem{CM1} \textsc{M. Colombo, G. Mingione}, \textit{Bounded Minimisers of double-phase
Variational Integrals}, Arch. Rational. Mech. Anal., 218 (2015), 219--173. 


\bibitem{CM} \textsc{M. Colombo, G. Mingione}, \textit{Regularity for double-phase Variational Problems}, Arch. Rational. Mech. Anal., 215 (2015), 443--496. 

\bibitem{CMM} \textsc{G. Cupini, P. Marcellini, E. Mascolo}, \textit{Regularity for Nonuniformly Elliptic Equations with $(p,q)$-Growth and Explicit $x,u$-Dependence}. Arch Rational Mech Anal 248, 60 (2024). \url{https://doi.org/10.1007/s00205-024-01982-0}.

\bibitem{DFM} \textsc{C. De Filippis, G. Mingione}, \textit{Manifold constrained non-uniformly elliptic problems}, J. Geom. Anal., 30 (2020), 1661--1723.	
	
\bibitem{DGV} \textsc{E. Di Benedetto, U. Gianazza, V. Vespri}, \textit{Remarks on local boundedness and local H\"older continuity of local weak solutions to anisotropic $p$-Laplacian type equations}, J. Elliptic Parabol. Equ., 2 (2016), 157--169.

\bibitem{DT}  \textsc{E. Di Benedetto and N.S. Trudinger}, \textit{Harnack inequality for quasiminima of variational integrals}, Ann. Inst. H. Poincar\'e Anal Non Lin\'eaire 1 (1984), 295--308.

\bibitem{DMV1} \textsc{G. D\"uzgun, P. Marcellini, V. Vespri}, \textit{An alternative approach to the H\"older continuity of solutions to some elliptic equations}, Nonlinear Anal., 94 (2014), 133--141.

\bibitem{DMV2} \textsc{G. D\"uzgun, P. Marcellini, V. Vespri}, \textit{Space expansion for a solution of an anisotropic $p$-Laplacian equation by using a parabolic approach}, Riv. Math. Univ. Parma, 5 (2014), 93--111.

\bibitem{Ele} \textsc{M. Eleuteri}, \textit{H\"older continuity results for a class of functionals with non-standard growth}, Bollettino dell'Unione Matematica Italiana, 7-B.1 (2004), 129--157.
	
	
\bibitem{ELM} \textsc{L. Esposito, F. Leonetti, G. Mingione}, \textit{Higher integrability for minimizers of integral functionals with (p, q) growth}, J. Diff. Equ.,  157 (1999), 414--438. 


\bibitem{FHK} \textsc{B. Franchi, P. Hajlasz, P. Koskela}, \textit{Definitions of Sobolev classes on metric spaces}, Ann. Inst. Fourier (Grenoble), 49 (1999), 1903--1924.


\bibitem{GG1} \textsc{M. Giaquinta, E. Giusti}, \textit{On the regularity of the minima of variational integrals}, Acta Math., 148 (1982), 31--46. 


\bibitem{GG2} \textsc{M. Giaquinta, E. Giusti}, \textit{Quasi-minima}, Ann. Inst. H. Poincar\'e Anal. Non Lin\'eaire, 1 (1984), 79--107.


    
\bibitem{G} \textsc{E. Giusti}, \textit{Direct Methods in the Calculus of Variations}, World Scientific Publishing, River Edge, NJ, (2003).


\bibitem{GAA}\textsc{A. A. Grigor’yan},  \textit{The heat equation on noncompact Riemannian manifolds}, Mat. Sb. 182 (1992), 55--87. Translation in Math. USSR-Sb. 72 (1992), 47--77.    
 
\bibitem{H} \textsc{P. Haj\l asz}, \textit{Sobolev spaces on metric-measure spaces}, Contemp. Math., 338 (2003), 173--218. 

\bibitem{HHT} \textsc{P. Harjulehto, P. Hästö, O. Toivanen}, \textit{H\"older regularity of quasiminimizers under generalized growth conditions}, Calc. Var., 56, 22 (2017).
 
    
\bibitem{HKST} \textsc{J. Heinonen, P. Koskela, N. Shanmugalingam, J. Tyson}, \textit{Sobolev Spaces on Metric Measure Spaces: An Approach Based on Upper Gradients}, New Mathematical Monographs, Cambridge University Press, 27, (2015).



\bibitem{Heran1}\textsc{A. Her\'an}, \textit{Harnack inequality for parabolic quasi minimizers on metric spaces}, Atti Accad. Naz. Lincei Rend. Lincei Mat. Appl., 31 (2020), 565--592.

\bibitem{Heran2}\textsc{A. Her\'an},  \textit{H\"older continuity of parabolic quasi-minimizers on metric measure spaces} Journal of Differential Equations 341 (2022), 208--262.

\bibitem{IK} \textsc{C. Irving, L. Koch}, \textit{Boundary regularity results for minimisers of convex functionals with $(p,q)$-growth}, Adv. Nonlinear Anal. 12, 20230110, (2023).    

\bibitem{IMM} \textsc{P.A. Ivert, N. Marola, M. Masson}, \textit{ Energy estimates for variational minimizers of a parabolic doubly nonlinear equation on metric measure spaces}, Ann. Acad. Sci. Fenn. Math., 39 (2014), 711--719.


\bibitem{KZ} \textsc{S. Keith, X. Zhong}, \textit{The Poincar\'{e} inequality is an open ended condition}, Ann. of Math., 167 (2008), 575--599. 

    
\bibitem{KKM} \textsc{T.  Kilpel\"ainen,  J.  Kinnunen,  O.  Martio}, \textit{Sobolev  Spaces  with  Zero  Boundary  Values  on Metric Spaces}, Potential Anal., 12 (2000), 233--247.


\bibitem{KLM} \textsc{J. Kinnunen, J. Lehrb\"{a}ck, A. V\"{a}h\"{a}kangas},	\textit{Maximal Function Methods for Sobolev Spaces}, Mathematical Surveys and Monographs, 257, American Mathematical Soc., (2021).


\bibitem{KMM} \textsc{J.  Kinnunen, N. Marola, O.  Martio}, \textit{Harnack's principle for quasiminimizers}, 	Ric. Mat., 56 (2007), 73--88.	

	
\bibitem{KM} \textsc{J. Kinnunen, O. Martio}, \textit{Potential theory of quasiminimizers}, Ann. Acad. Sci. Fenn. Math., 28 (2003), 459--490.

\bibitem{KNP} \textsc{J. Kinnunen, A. Nastasi, C. Pacchiano Camacho}, \textit{Gradient higher integrability for double-phase problems on metric measure spaces}, 	 Proc. Amer. Math. Soc. Volume 152, Number 3, (2024) 1233--1251.

\bibitem{KS} \textsc{J. Kinnunen, N. Shanmugalingam}, \textit{Regularity of quasi-minimizers on metric spaces}, Manuscripta Math., 105 (2001), 401--423.

\bibitem{Lieberman} \textsc{G.M. Lieberman}, \textit{The natural generalization of the natural conditions of Ladyzhenskaya and Ural'tseva for elliptic equations}, Commun. PDE, 16 (1991), 311--361.



\bibitem{LSZ} \textsc{Q. Liu, N. Shanmugalingam, X. Zhou}, \textit{Discontinuous eikonal equations in metric measure spaces},
arxiv: \url{https://arxiv.org/abs/2308.06872}, to appear in Trans. AMS.

\bibitem{Ma0} \textsc{P. Marcellini}, \textit{Anisotropic and $p, q$-nonlinear partial differential equations}, Rend. Lincei Sci. Fis. Nat., 31 (2020), 295–301.	


\bibitem{Ma} \textsc{P. Marcellini}, \textit{Local Lipschitz continuity for $p,q$-PDEs with explicit $u$-dependence}, Nonlinear Analysis, 226, (2023).

\bibitem{Ma1} \textsc{P. Marcellini}, \textit{Regularity and existence of solutions of elliptic equations with $p$, $q$-growth conditions}, J. Differential Equations, 90 (1991), 1--30.


\bibitem{Ma2} \textsc{P. Marcellini}, \textit{Regularity for elliptic equations with general growth conditions}, J. Differential Equations, 105 (1993), 296--333.


\bibitem{Ma3} \textsc{P. Marcellini}, \textit{Regularity of minimizers of integrals of the calculus of variations with non standard growth conditions}, Arch. Rational Mech. Anal., 105 (1989), 267--284.

\bibitem{MNPC} \textsc{P. Marcellini, A. Nastasi, C. Pacchiano Camacho} \textit{Unified a-priori estimates for minimizers under $p,q$-growth and exponential growth}, submitted. arxiv: \url{https://arxiv.org/abs/2410.22875}.

\bibitem{MM}\textsc{ N. Marola, M. Masson}, \textit{On the Harnack inequality for parabolic minimizers in metric measure spaces}, Tohoku Math. J., 65 (2013), 569--589. 

\bibitem{MMPP} \textsc{M. Masson, M. Miranda Jr., F. Paronetto, M. Parviainen}, \textit{Local higher integrability for parabolic quasiminimizers in metric spaces}, Ric. Mat., 62 (2013), 279--305.

\bibitem{MP} \textsc{M. Masson, M. Parviainen}, \textit{Global higher integrability for parabolic quasiminimizers in metric measure spaces}, J. Anal. Math., 126 (2015), 307–339.

\bibitem{MS} \textsc{M. Masson, J. Siljander}, \textit{H\"older regularity for parabolic De Giorgi classes in metric measure spaces}, Manuscr. Math., 142 (2013), 187–214.


\bibitem{Mazya2} \textsc{V. G. Maz'ya},  \textit{On the continuity at a boundary point of solutions of quasi-linear elliptic equations}, Vestnik Leningrad. Univ. Mat. Mekh. Astronom. 25 (1970), 42--55 (Russian).
English transl.: Vestnik Leningrad Univ. Math. 3 (1976), 225--242.


\bibitem{Mazya1} \textsc{V. G. Maz'ya},  \textit{Regularity at the boundary of solutions of elliptic equations and conformal mapping}, Dokl. Akad. Nauk SSSR 150 (1963), 1297--1300 (Russian). English transl.: Soviet Math. Dokl. 4 (1963), 1547--1551.


\bibitem{Maz} \textsc{V. G. Maz'ya}, \textit{Sobolev spaces}, Springer-Verlag, Berlin-New York, (1985).



\bibitem{Min} \textsc{G. Mingione}, \textit{Regularity of minima: an invitation to the dark side of the calculus of variations}, Appl. Math. 51 (2006), 355--425.


\bibitem{NP} \textsc{A. Nastasi, C. Pacchiano Camacho}, \textit{Regularity properties for quasiminimizers of a $(p,q)$-Dirichlet integral}, Calc. Var., 227  (2021), 1--37.


\bibitem{PCN2} \textsc{A. Nastasi, C. Pacchiano Camacho}, \textit{Regularity results for quasiminima of a class of double-phase problems}, Math. Ann., (2024).

\bibitem{O} \textsc{J. Ok},
\textit{Regularity for double phase problems under additional integrability assumptions},
Nonlinear Analysis, 194 (2020), 111408.


\bibitem{SC}\textsc{ L.Saloff-Coste}, \textit{A note on Poincar\'e, Sobolev and Harnack inequalities}, Int. Math. Res. Not. IMRN 2 (1992), 27--38.

\bibitem{Sh1} \textsc{N. Shanmugalingam}, \textit{Newtonian spaces: an extension of Sobolev spaces to metric measure spaces}, Rev. Mat. Iberoam., 16 (2000), 243--279.

\bibitem{Tach} \textsc{A. Tachikawa}, \textit{Boundary regularity of minimizers of double-phase functionals}, J. Math. Anal. Appl., 501 (2021).

\bibitem{Z1} \textsc{V.V. Zhikov}, \textit{Averaging of functionals of the calculus of variations and elasticity theory}, Izv. Akad. Nauk SSSR Ser. Mat.,
50 (4) (1986), 675--710.

\bibitem{Z2} \textsc{V.V. Zhikov}, \textit{On Lavrentiev’s phenomenon}, Russian J. Math. Phys., 3 (2) (1995), 249--269.

\bibitem{Z3} \textsc{V.V. Zhikov, S.M. Kozlov, O.A. Oleinik}, \textit{Homogenization of Differential Operators and Integral Functionals}, Springer-Verlag, Berlin, 1994.

\bibitem{Ziemer} \textsc{W.P. Ziemer},  \textit{Boundary regularity for quasiminima}, Arch. Ration. Mech. Anal., 92 (1986), 371--382.

\end{thebibliography}
\end{document}